 \documentclass[10pt, reqno]{amsart}
\usepackage[a4paper,margin=1.20in]{geometry}

\usepackage{palatino}
\usepackage{amsmath,amsthm,amssymb,mathrsfs}
\usepackage[abbrev,non-sorted-cites]{amsrefs}
\usepackage{color}
\usepackage{verbatim}
\usepackage{datetime}
\usepackage{hyperref}

\title[Marginally trapped surfaces]{{\small{Marginally trapped surfaces in ${\mathbb{L}}^{4}$ and three Weierstrass representations}}}

\author{Hojoo Lee}
\address{Department of Mathematics and Institute of Pure and Applied Mathematics, Jeonbuk National University,
Jeonju 54896, Korea}
\email{compactkoala@gmail.com, kiarostami@jbnu.ac.kr}

\newtheorem{thm}{Theorem} 
\newtheorem{lem}[thm]{Lemma}
\newtheorem{cor}[thm]{Corollary}
\newtheorem{prop}[thm]{Proposition}

\theoremstyle{definition}
\newtheorem{rem}{Remark}[section]
\newtheorem{exmp}{Example}[section]

\newtheorem{defn}{Definition}[section]

\numberwithin{equation}{section}

\begin{document}



\begin{abstract}
 We construct new integrable systems to present Weierstrass type representations for spacelike surfaces whose mean 
 curvature vector $\mathbf{H}$ satisfies the null condition $\langle \mathbf{H},  \mathbf{H} \rangle=0$ in the four dimensional Lorentz-Minkowski space ${\mathbb{L}}^{4}$.  
 Our new Weierstrass presentations extend simultaneously classical Weierstrass representations (of the first kind and the second kind) for maximal surfaces in ${\mathbb{L}}^{3}$ and 
 minimal surfaces in ${\mathbb{R}}^{3}$. We solve a linear partial differental equation to construct explicit examples of marginally trapped surfaces with nowhere vanishing 
 mean curvature vector.
\end{abstract}

\maketitle

  \tableofcontents
 
  \section{Motivation and main results}
  
  A spacelike space in the four dimensional Lorentz-Minkowski space ${\mathbb{L}}^{4}$ is marginally trapped
 if its mean curvature vector $\mathbf{H}$ verifies the null condition:
  \begin{equation*} 
      \langle \mathbf{H},  \mathbf{H} \rangle=0.
  \end{equation*} 
  We briefly review known conformal representations \cites{AGM2005, DFS2021, Liu 2013} for marginally trapped surfaces.    
  
The Aledo-G\'{a}lvez-Mira representation \cite[Theorem 9]{AGM2005} generates marginally trapped surfaces of \emph{Bryant type} in ${\mathbb{L}}^{4}$. Their 
conformal representation yields marginally trapped surfaces satisfying two geometric conditions that they have flat normal bundle and are locally isometric to 
some minimal surfaces in Euclidean space ${\mathbb{R}}^{3}$ or some maximal surfaces in  Lorentz-Minkowski space ${\mathbb{L}}^{3}$. The Aledo-G\'{a}lvez-Mira 
representation extends simultaneously the Bryant representation \cite[Theorem A]{Bryant 1987} for surfaces with mean curvature one in hyperbolic space ${\mathbb{H}}^{3} \subset {\mathbb{L}}^{4}$ and the representation \cite[Theorem 7.1]{AK1999} for spacelike surfaces with mean curvature one in de-Sitter space ${\mathbb{S}}^{3}_{1} \subset {\mathbb{L}}^{4}$.

 Recently, Dussan, Filho, Sim\~{o}es \cite[Section 5]{DFS2021} presented an interesting method to construct explicit examples of marginally trapped surfaces sitting in ${\mathbb{H}}^{3} \subset {\mathbb{L}}^{4}$. They solved a Riccati partial differential equation for the triple of a holomorphic function and two $\mathbb{C}$-valued functions. The conformal representation due to Liu \cite[Theorem 4.2]{Liu 2013} for general marginally trapped surfaces in ${\mathbb{L}}^{4}$ requires to solve a system of partial differential 
equations for the triple of a holomorphic function and two $\mathbb{C}$-valued functions. Our work is strongly motivated by the Liu representation.
    
 We sketch the main results of this paper, the first one in a series of papers on the marginally trapped surfaces.   We present the three Weierstrass type 
  representations (Theorem \ref{Poisson 01}, Corollary \ref{Poisson 02}, Corollary \ref{Poisson 03}) for marginally trapped surfaces in the four dimensional 
  Lorentz-Minkowski space ${\mathbb{L}}^{4}$. Our representations require to solve the \emph{linear} partial differential equations.  
  
  Inspired by the Liu integrable system (\cite[Theorem 4.2]{Liu 2013} and Lemma \ref{Liu}) for 
  the triple of a holomorphic function and two $\mathbb{C}$-valued functions, we construct integrable systems (Definition \ref{FWD01} and Defintion \ref{FWD02}) for 
  the triple of a holomorphic function and two $\mathbb{R}$-valued functions. The fundamental features of our integrable systems are investigated in 
  Section \ref{integrable systems} and Section \ref{parameters}.
  
  The key geometric idea hidden in our algebraic reduction of the Liu integrable system is revealed in our Weierstrass type representation (in Corollary \ref{Poisson 02}) for marginally trapped 
  surfaces in ${\mathbb{L}}^{4}$. Our conformal representation naturally extends the Weierstrass type representation of  \emph{the second kind} (due to O. 
  Kobayashi \cite[Corollary 1.3]{Kobayashi 1983}) for maximal surfaces in ${\mathbb{L}}^{3}$. (See Remark \ref{Poisson 02 Remarks}.)
  
 Let $h(u+iv)$ be an arbitrary prescribed nowhere vanishing holomorphic function. Our conformal representation in Corollary \ref{Poisson 02} indicates that the pair  $\left(\mathcal{M}(u, v), \mathcal{N}(u, v) \right)$ of the $\mathbb{R}$-valued 
  functions, which solves the \emph{linear} partial differential equation  
    \begin{equation*}   \label{IntroP01} 
       {\mathcal{M}}_{uu} +  {\mathcal{M}}_{vv}= \textrm{Re} \,  h  \,  \left(  {\mathcal{N}}_{uu} +  {\mathcal{N}}_{vv}\right) , 
   \end{equation*}  
   yields the marginally trapped surface in ${\mathbb{L}}^{4}$. (The holomorphic function $g:=\frac{1}{\,h\,}$ 
 can be viewed as a complexified null Gauss map of the marginally trapped surface. See Remark \ref{Poisson 02 Remarks}.) Our representation in 
   Corollary \ref{Poisson 03} can be viewed a natural generalization of classical representations of \emph{the first kind} for maximal surfaces
    in ${\mathbb{L}}^{3}$ and minimal surfaces in ${\mathbb{R}}^{3}$. (See Remark \ref{Poisson 03 Remarks}.)

    Solving the above linear partial differential equation yields explicit examples of marginally trapped surfaces in ${\mathbb{L}}^{4}$. In particular, in 
    Section \ref{deformed examples}, we construct two families of marginally trapped surfaces with nowhere vanishing mean curvature vector. Both families 
    contain a \emph{catenoid cousin} with mean curvature one in the three dimensional hyperbolic space ${\mathbb{H}}^{3} \subset {\mathbb{L}}^{4}$.

  \section{Two equivalent integrable systems} \label{integrable systems}
 
 Throughout this section, $\Omega \subset {\mathbb{R}}^2 \equiv \mathbb{C}$ denotes a simply connected domain with the complex coordinate $z$. 
 
 \begin{defn}[\textbf{Weierstrass data of the first kind}] \label{FWD01}
We call the triple $\left(g, \, \mathcal{P}, \, \mathcal{Q}\right)$ the Weierstrass data of the first kind if a function $g : \Omega \to \mathbb{C}-\left\{0\right\}$ and two ${\mathcal{C}}^{2}$ functions  $\mathcal{P},  \mathcal{Q}: \Omega \to \mathbb{R}$ satisfy the three conditions
 \begin{equation*}  \label{FW01}
 \begin{cases}
    g_{\overline{\,z\,} } =0, \\
    {\mathcal{P}}_{z \overline{\,z\,}} =  {\vert \, g \, \vert}^{2}  \;  {\mathcal{Q}}_{z \overline{\,z\,}}, \\ 
    {\mathcal{P}}_{z}  -  {\vert \, g \, \vert}^{2} \,  {\mathcal{Q}}_{z} \neq 0.
\end{cases}    
 \end{equation*}
\end{defn}

\begin{defn}[\textbf{Weierstrass data of the second kind}] \label{FWD02}
We call the triple $\left(h, \, \mathcal{M}, \, \mathcal{N}\right)$ the Weierstrass data of the second kind if a function $h : \Omega \to \mathbb{C}-\left\{0\right\}$ and two ${\mathcal{C}}^{2}$ functions  $\mathcal{M},  \mathcal{N}: \Omega \to \mathbb{R}$ satisfy the three conditions
 \begin{equation*}  \label{FW02}
 \begin{cases}
    h_{\overline{\,z\,} } =0, \\
   {\mathcal{M}}_{ z \overline{\,z\,} }= \left( \textrm{Re} \,  h \right)  \,  {\mathcal{N}}_{ z \overline{\,z\,} }, \\ 
  {\mathcal{M}}_{z} - \left( \textrm{Re} \,  h \right)  \,  { \mathcal{N}}_{z} \neq 0.
\end{cases}    
 \end{equation*}
\end{defn}

 \begin{thm}[\textbf{Equivalence transformation from the first data to the second data}] \label{EAmain}
Let $\left(g, \, \mathcal{P}, \, \mathcal{Q}\right)$ be a Weierstrass data of the first kind. We define $h:=\frac{1}{\,g\,}$ and $\mathcal{N}:=2 \mathcal{P}$. 
 Then, there exists a ${\mathcal{C}}^{2}$ function ${\mathcal{M}}_{\lambda}: \Omega \to \mathbb{R}$ such that 
 \begin{equation*}  \label{EA01}
      {\mathcal{M}}_{z} = \frac{1}{\, g \,} \,  {\mathcal{P}}_{z} + g \, {\mathcal{Q}}_{z}.  
 \end{equation*}
 Moreover, the triple $\left(h, \, \mathcal{M}, \, \mathcal{N}\right)$ becomes a Weierstrass data of the second kind such that
  \begin{equation*}  \label{EA02}
    \, {\mathcal{M}}_{z} -   \left( \textrm{Re} \,  h \right)  \,  { \mathcal{N}}_{z} \, = 
     -  \frac{1}{\, \overline{\,g\,} \,} \left( \,  {\mathcal{P}}_{z}  -  {\vert \, g \, \vert}^{2} \,  {\mathcal{Q}}_{z} \, \right).
  \end{equation*}   
\end{thm}

\begin{proof}
The function $g : \Omega \to \mathbb{C}-\left\{0\right\}$ and two ${\mathcal{C}}^{2}$ functions  $\mathcal{P},  \mathcal{Q}: \Omega \to \mathbb{R}$ satisfy  
 \begin{equation*}  \label{EAp02}
    g_{\overline{\,z\,} } =0, \quad
    {\mathcal{P}}_{z \overline{\,z\,}} =  {\vert \, g \, \vert}^{2}  \;  {\mathcal{Q}}_{z \overline{\,z\,}}, \quad 
    {\mathcal{P}}_{z}  -  {\vert \, g \, \vert}^{2} \,  {\mathcal{Q}}_{z} \neq 0.    
 \end{equation*}
It is immediate that $h_{\overline{\,z\,} }= {\left( \, \frac{1}{\,g\,}\, \right)}_{\overline{\,z\,} } =0$. We introduce the function $  \mathcal{E}: \Omega \to \mathbb{R}$ by 
 \begin{equation*}  \label{EAp03}
  \mathcal{E}:=  \frac{1}{\, g \,} \,  {\mathcal{P}}_{z} + g \, {\mathcal{Q}}_{z}.
 \end{equation*}
We have 
 \begin{equation*}  \label{EAp04}
      {\mathcal{E}}_{ \overline{\,z\,} } =  \frac{1}{\, g \,} \,  {\mathcal{P}}_{z \overline{\,z\,}} + g \, {\mathcal{Q}}_{z \overline{\,z\,}}
      =  \left(\, \frac{1}{\, g \,} \,     + \frac{1}{\,\overline{\, g\,} \,} \, \right) \, {\mathcal{P}}_{z \overline{\,z\,}}
      =  2 \left( \textrm{Re} \,  h \right)  {\mathcal{P}}_{z \overline{\,z\,}}    \in \mathbb{R}.
 \end{equation*}
 Since the domain $\Omega$ is simply connected, by Poincar\'{e}'s Lemma, this 
    implies the existence of an 
 ${\mathbb{R}}$-valued function ${\mathcal{M}}_{\lambda}$ defined on $\Omega$ such that
 \begin{equation*}  \label{EAp05}
       {\mathcal{M}}_{z} =  \mathcal{E}.
 \end{equation*}       
We have 
 \begin{equation*}  \label{EAp06}
       {\mathcal{M}}_{z  \overline{\,z\,} } =  {\mathcal{E}}_{ \overline{\,z\,} }   =  2 \left( \textrm{Re} \,  h \right)  {\mathcal{P}}_{z \overline{\,z\,}}  = \left(  \textrm{Re} \,  h  \, \right) {\mathcal{N}}_{z \overline{\,z\,}} 
 \end{equation*}     
 and  
 \begin{eqnarray*}  \label{EAp07}
   {\mathcal{M}}_{z} -   \left( \textrm{Re} \,  h \right)  \,  { \mathcal{N}}_{z} 
    &=&  \mathcal{E} - \frac{1}{\,2\,} \left(  \frac{1}{\, g \,} \,     + \frac{1}{\,\overline{\, g\,} \,}  \right)  2  { \mathcal{P}}_{z}  
   =  \frac{1}{\, g \,} \,  {\mathcal{P}}_{z} + g \, {\mathcal{Q}}_{z} - \left(  \frac{1}{\, g \,} \,     + \frac{1}{\,\overline{\, g\,} \,}  \right)  { \mathcal{P}}_{z} \\
     &=& - \frac{1}{\, \overline{\,g\,} \,} \left( \,  {\mathcal{P}}_{z}  -  {\vert \, g \, \vert}^{2} \,  {\mathcal{Q}}_{z} \, \right) \neq 0.
 \end{eqnarray*}
 We conclude that  the triple $\left(h, \, \mathcal{M}, \, \mathcal{N}\right)$ is a Weierstrass data of the second kind:
 \begin{equation*}  \label{EAp08}
    h_{\overline{\,z\,} } =0, \quad
   {\mathcal{M}}_{ z \overline{\,z\,} }=  \left( \textrm{Re} \,  h \right)   \,  {\mathcal{N}}_{ z \overline{\,z\,} }, \quad 
  {\mathcal{M}}_{z} -  \left( \textrm{Re} \,  h \right) \,  { \mathcal{N}}_{z} \neq 0.  
 \end{equation*}
\end{proof}

 \begin{thm}[\textbf{Equivalence transformation from the second data to the first data}]  \label{EBmain}
Let $\left(h, \, \mathcal{M}, \, \mathcal{N}\right)$ be a Weierstrass data of the second kind. We define $g:=\frac{1}{\,h\,}$ and $\mathcal{P}:=\frac{1}{\,2\,} \mathcal{N}$. 
 Then, there exists a ${\mathcal{C}}^{2}$ function ${\mathcal{Q}}_{\lambda}: \Omega \to \mathbb{R}$ such that 
 \begin{equation*}  \label{EB01a}
      {\mathcal{Q}}_{z} = h \,  {\mathcal{M}}_{z} - \frac{\, h^2 \,}{2} {\mathcal{N}}_{z}.  
 \end{equation*}
 Moreover, the triple $\left(g, \, \mathcal{P}, \, \mathcal{Q}\right)$ becomes a Weierstrass data of the first kind such that 
  \begin{equation*}  \label{EB01b}
  \frac{1}{\, \overline{\,g\,} \,} \left( \,  {\mathcal{P}}_{z}  -  {\vert \, g \, \vert}^{2} \,  {\mathcal{Q}}_{z} \, \right)
  = - \left( \,  {\mathcal{M}}_{z} -   \left( \textrm{Re} \,  h \right)  \,  { \mathcal{N}}_{z} \, \right).
  \end{equation*}      
\end{thm}

\begin{proof}
The function $h : \Omega \to \mathbb{C}-\left\{0\right\}$ and two ${\mathcal{C}}^{2}$ functions  $\mathcal{M},  \mathcal{N}: \Omega \to \mathbb{R}$ satisfy  
 \begin{equation*}  \label{EBp02}
    h_{\overline{\,z\,} } =0, \quad
   {\mathcal{M}}_{ z \overline{\,z\,} }= \frac{\, h + \overline{\, h\,} \,}{2}  \,  {\mathcal{N}}_{ z \overline{\,z\,} }, \quad 
  {\mathcal{M}}_{z} - \frac{\, h + \overline{\, h\,} \,}{2} \,  { \mathcal{N}}_{z} \neq 0.  
 \end{equation*}
It is immediate that $g_{\overline{\,z\,} }= {\left( \, \frac{1}{\,h\,}\, \right)}_{\overline{\,z\,} } =0$. We introduce the function $  \mathcal{E}: \Omega \to \mathbb{R}$ by 
 \begin{equation*}  \label{EBp03}
  \mathcal{E}:= h \, {\mathcal{M}}_{z} -  \frac{\, h^2 \,}{2} {\mathcal{N}}_{z}.
 \end{equation*}
We have 
 \begin{equation*}  \label{EBp04}
      {\mathcal{E}}_{ \overline{\,z\,} } = h \, {\mathcal{M}}_{z  \overline{\,z\,} } -  \frac{\, h^2 \,}{2} {\mathcal{N}}_{z  \overline{\,z\,} }
            = \left(  h \frac{\, h + \overline{\, h\,} \,}{2} -  \frac{\, h^2 \,}{2}  \right) {\mathcal{N}}_{z  \overline{\,z\,} }
            = \frac{\, {\vert \, h \, \vert}^{2} \,}{2} {\mathcal{N}}_{z  \overline{\,z\,} } \in \mathbb{R}.
 \end{equation*}
 Since the domain $\Omega$ is simply connected, by Poincar\'{e}'s Lemma, this 
    implies the existence of an 
 ${\mathbb{R}}$-valued function ${\mathcal{Q}}_{\lambda}$ defined on $\Omega$ such that
 \begin{equation*}  \label{EBp05}
       {\mathcal{Q}}_{z} =  \mathcal{E}.
 \end{equation*}       
We have 
 \begin{equation*}  \label{EBp06}
       {\mathcal{Q}}_{z  \overline{\,z\,} } =  {\mathcal{E}}_{ \overline{\,z\,} } =  \frac{\, {\vert \, h \, \vert}^{2} \,}{2} {\mathcal{N}}_{z  \overline{\,z\,} }
       = \frac{1}{\, {\vert \, g \, \vert}^{2} \,} {\mathcal{P}}_{z  \overline{\,z\,} }
 \end{equation*}     
 and  
 \begin{equation*}  \label{EBp07}
  \frac{1}{\, {\vert \, g \, \vert}^{2} \,} \,  {\mathcal{P}}_{z}  -  {\mathcal{Q}}_{z} 
 =  \frac{\, {\left\vert \, h \, \right\vert}^{2} \,}{\,2\,}  {\mathcal{N}}_{z}  -   \, \left( \, h \, {\mathcal{M}}_{z} -  \frac{\, h^2 \,}{2} {\mathcal{N}}_{z} \, \right) 
= h \left( \,  \left( \textrm{Re} \,  h \right)  \,  { \mathcal{N}}_{z} -   {\mathcal{M}}_{z}  \, \right) \neq 0.
 \end{equation*}
 We conclude that  the triple $\left(g, \, \mathcal{P}, \, \mathcal{Q}\right)$ is a Weierstrass data of the first kind:
 \begin{equation*}  \label{EBp08}
    g_{\overline{\,z\,} } =0, \quad
    {\mathcal{P}}_{z \overline{\,z\,}} =  {\vert \, g \, \vert}^{2}  \;  {\mathcal{Q}}_{z \overline{\,z\,}}, \quad 
    {\mathcal{P}}_{z}  -  {\vert \, g \, \vert}^{2} \,  {\mathcal{Q}}_{z} \neq 0.    
 \end{equation*}
\end{proof}

\begin{rem} \label{equivalence of two systems}
 It is straightforward to check that the transformation in  Theorem \ref{EBmain} is the reverse transformation in Theorem \ref{EAmain} using the integrability conditions.
\end{rem}

 We present three fundamental deformations of Weierstrass data of the first kind and the second kind.  
 
\begin{prop}[\textbf{Parabolic Deformations}] \label{FIS01}
 Let the triple $\left(g, \, \mathcal{P}, \, \mathcal{Q}\right)$ be a Weierstrass data of the first kind. Let $\lambda \in \mathbb{R}$ be a parameter constant 
 such that $g(z) \neq - \frac{1}{\,  i \lambda   \, } $ for all $z \in \Omega$. We set $g_{ {}_{\lambda} } := \frac{g}{\, 1 + i \lambda g\,}$ and ${\mathcal{P}}_{\lambda} :=\mathcal{P}$.  
 Then, there exists a ${\mathcal{C}}^{2}$ function ${\mathcal{Q}}_{\lambda}: \Omega \to \mathbb{R}$ satisfying the integrability condition 
 \begin{equation*}  \label{FI01p01}
      {\left( \, {\mathcal{Q}}_{\lambda} \, \right)}_{z} = \left( \,  \frac{1}{\,g\,} + i \lambda \, \right) \left(  g    {\mathcal{Q}}_{z} - i \lambda {\mathcal{P}}_{z}      \right).  
 \end{equation*}
 Moreover, the triple $\left(g_{ {}_{\lambda} }, \,{\mathcal{P}}_{\lambda}, \, {\mathcal{Q}}_{\lambda}\right)$ becomes a Weierstrass data of the first kind. 
\end{prop}

\begin{proof}[\textbf{First Proof of Proposition \ref{FIS01}}]  We observe that both $g_{ {}_{\lambda}}$ and $g=g_{ {}_{0} }$ vanishes nowhere. We have 
 \begin{equation*}  \label{FI01p02}
      \frac{1}{\,g_{ {}_{\lambda} }\,} = \frac{1}{\,g\,} + i \lambda \quad \text{and} \quad 
         \frac{1}{\, \overline{\,g_{ {}_{\lambda} }\,} \,} = \frac{1}{\,\overline{g}\,}  -i \lambda.
 \end{equation*}
We introduce the function $  \mathcal{E}: \Omega \to \mathbb{R}$ by 
 \begin{equation*}  \label{FI01p03}
  \mathcal{E}:=\left( \,  \frac{1}{\,g\,} + i \lambda \, \right) \left(  g    {\mathcal{Q}}_{z} - i \lambda {\mathcal{P}}_{z}      \right) = 
 \frac{1}{\,g_{ {}_{\lambda} }\,}   \left(  g    {\mathcal{Q}}_{z} - i \lambda {\mathcal{P}}_{z}      \right).
 \end{equation*}
Note that both $g_{ {}_{\lambda}}$ and $g=g_{ {}_{0} }$ are holomorphic and that ${\mathcal{Q}}_{z \overline{\,z\,}}= \frac{1}{\,  {\vert \, g \, \vert}^{2}  \, } {\mathcal{P}}_{z \overline{\,z\,}}$. We have
 \begin{eqnarray*}  \label{FI01p04}
      {\mathcal{E}}_{ \overline{\,z\,} } &=&   \frac{g}{\,g_{ {}_{\lambda} }\,}     {\mathcal{Q}}_{z \overline{\,z\,}} +  \frac{1}{\,g_{ {}_{\lambda} }\,}  \left( \,  - i \lambda   \, \right) {\mathcal{P}}_{z  \overline{\,z\,} }  
    =   \frac{g}{\,g_{ {}_{\lambda} }\,}  \frac{\, {\mathcal{P}}_{z \overline{\,z\,}} \,}{\,  g \overline{\,g\,}  \, } +  \frac{1}{\,g_{ {}_{\lambda} }\,}  \left( \,  - i \lambda   \, \right) {\mathcal{P}}_{z  \overline{\,z\,} }  \\
        &=& \frac{1}{\,g_{ {}_{\lambda} }\,}   \left(\,   \frac{\, 1 \,}{\,  \overline{\,g\,}  \, }  - i \lambda   
 \, \right)  {\mathcal{P}}_{z  \overline{\,z\,} }  
 =  \frac{1}{\, {\left\vert \,   g_{ {}_{\lambda} } \,  \right\vert}^{2} \,}    {\mathcal{P}}_{z  \overline{\,z\,} }.
 \end{eqnarray*}
Since the domain $\Omega$ is simply connected, by Poincar\'{e}'s Lemma, the observation 
 \begin{equation*}  \label{FI01p05}
  {\mathcal{E}}_{ \overline{\,z\,} } =\frac{1}{\, {\left\vert \,   g_{ {}_{\lambda} } \,  \right\vert}^{2} \,}    {\mathcal{P}}_{z  \overline{\,z\,} } \in \mathbb{R}
   \end{equation*} 
    implies the existence of an 
 ${\mathbb{R}}$-valued function ${\mathcal{Q}}_{\lambda}$ defined on $\Omega$ such that
 \begin{equation*}  \label{FI01p06}
        {\left( \, {\mathcal{Q}}_{\lambda} \, \right)}_{z}  =  \mathcal{E}.
 \end{equation*}           
It follows that  
 \begin{equation*}  \label{FI01p07}
          {\left( \, {\mathcal{Q}}_{\lambda} \, \right)}_{z  \overline{\,z\,} }   =  {\mathcal{E}}_{\overline{\,z\,} }
          =  \frac{1}{\, {\left\vert \,   g_{ {}_{\lambda} } \,  \right\vert}^{2} \,}    {\mathcal{P}}_{z  \overline{\,z\,} }
          =   \frac{1}{\, {\left\vert \,   g_{ {}_{\lambda} } \,  \right\vert}^{2} \,}    { \left( {\mathcal{P}}_{\lambda} \right) }_{z  \overline{\,z\,} }.
 \end{equation*}    
 Finally, we have the equality
 \begin{eqnarray*}  \label{FI08}
    \frac{1}{\, \overline{\, g_{ {}_{\lambda} } \,}  \,}   {\left( \, {\mathcal{P}}_{\lambda} \, \right)}_{z}  -    g_{ {}_{\lambda} }   \,  {\left( \, {\mathcal{Q}}_{\lambda} \, \right)}_{z}
   &=&       \frac{1}{\, \overline{\, g_{ {}_{\lambda} } \,}  \,}   {\left( \, {\mathcal{P}}  \, \right)}_{z}  -    g_{ {}_{\lambda} }   \,  \mathcal{E}  
  =    \frac{1}{\, \overline{\, g_{ {}_{\lambda} } \,}  \,}   {\left( \, {\mathcal{P}}  \, \right)}_{z}  -    g_{ {}_{\lambda} }   \,  
    \left( \,    \frac{1}{\,   g_{ {}_{\lambda} }   \,}  \left(  g    {\mathcal{Q}}_{z} - i \lambda {\mathcal{P}}_{z}      \right)   \, \right) \\
    &=&  \left( \,    \frac{1}{\, \overline{\, g_{ {}_{\lambda} } \,}  \,}  + i \lambda    \, \right)  {\mathcal{P}}_{z}  -    g   \,  {\mathcal{Q}}_{z} 
    =  \frac{1}{\, \overline{\, g\,}  \,}   {\mathcal{P}}_{z}  -    g   \,  {\mathcal{Q}}_{z}, 
 \end{eqnarray*}
This implies the last condition 
 \begin{equation*}  \label{FI01p09}
      {\left( \,  {\mathcal{P}}_{\lambda}  \, \right) }_{z}  -  {\vert \, g_{ {}_{\lambda} }  \, \vert}^{2} \,  {\left(\, {\mathcal{Q}}_{\lambda} \, \right)}_{z} 
  =   \frac{\,   \overline{\, g_{ {}_{\lambda} } \,}     \,}{\,  \overline{\, g\,}   \,}  \left( {\mathcal{P}}_{z}  -  {\vert \, g \, \vert}^{2} \,  {\mathcal{Q}}_{z} \right) \neq 0.
 \end{equation*}  
\end{proof}

Though the first proof of Proposition \ref{FIS02} is self-contained, it does not explain the method how to discover the integrability condition in Proposition \ref{FIS02}:
 \begin{equation*}  \label{FI01p01again}
      {\left( \, {\mathcal{Q}}_{\lambda} \, \right)}_{z} = \left( \,  \frac{1}{\,g\,} + i \lambda \, \right) \left(  g    {\mathcal{Q}}_{z} - i \lambda {\mathcal{P}}_{z}      \right).  
 \end{equation*}
We use the equivalence of the Weierstrass data of the first kind and the Weierstrass data of the second kind to discover the integrability condition. 

\begin{proof}[\textbf{Second Proof of Proposition \ref{FIS01}}] 
We set $h=\frac{1}{\,g\,}$ and $\mathcal{N}=2 \mathcal{P}$. According to Theorem \ref{EAmain}, we can
 find a ${\mathcal{C}}^{2}$ function ${\mathcal{M}}: \Omega \to \mathbb{R}$ such that 
 \begin{equation*}  \label{FI01q01}
             {\mathcal{M}}_{z} = \frac{1}{\, g \,} \,  {\mathcal{P}}_{z} + g \, {\mathcal{Q}}_{z}. 
 \end{equation*}
 Moreover, the triple $\left(h, \, \mathcal{M}, \, \mathcal{N}\right)$ is a Weierstrass data of the second kind: 
  \begin{equation*} \label{FI01q02}
    h_{\overline{\,z\,} } =0, \quad
   {\mathcal{M}}_{ z \overline{\,z\,} }=  \left( \textrm{Re} \,  h \right)   \,  {\mathcal{N}}_{ z \overline{\,z\,} }, \quad 
  {\mathcal{M}}_{z} -  \left( \textrm{Re} \,  h \right) \,  { \mathcal{N}}_{z} \neq 0.  
 \end{equation*}
It follows from $\lambda \in \mathbb{R}$ that $ \textrm{Re} \left( h +i \lambda \right)= \textrm{Re} \,  h$.
We find immediately that the triple 
\begin{equation*} \label{FI01q03}
  \left(h_{{}_{\lambda}}, \, {\mathcal{M}}_{\lambda}, \, {\mathcal{N}}_{\lambda }\right)=\left(h + i \lambda, \, \mathcal{M}, \, \mathcal{N}\right)
 \end{equation*}
is also a Weierstrass data of the second kind:
  \begin{equation*} \label{FI01q04}
    {\left(\, h_{{}_{\lambda}} \, \right)}_{\overline{\,z\,} } =0, \quad
   {\left(\, {\mathcal{M}}_{\lambda} \, \right)}_{ z \overline{\,z\,} }=  \left( \textrm{Re} \,  h_{{}_{\lambda}} \right)   \,     {\left(\, {\mathcal{N}}_{\lambda} \, \right)}_{ z \overline{\,z\,} }, \quad 
   {\left(\, {\mathcal{M}}_{\lambda} \, \right)}_{z} -  \left( \textrm{Re} \,  h_{{}_{\lambda}} \right) \,     {\left(\, {\mathcal{N}}_{\lambda} \, \right)}_{z} \neq 0.  
 \end{equation*}
   Theorem \ref{EBmain} guarantees the existence of a ${\mathcal{C}}^{2}$ function ${\mathcal{Q}}_{\lambda}: \Omega \to \mathbb{R}$ such that 
 \begin{equation*}  \label{FI01q05}
  {\left(\, {\mathcal{Q}}_{\lambda} \,\right)}_{z} = h_{{}_{\lambda}}  \,  {\left(\, {\mathcal{M}}_{\lambda} \,\right)}_{z} - \frac{\, {h_{{}_{\lambda}} }^2 \,}{2} {\left(\, {\mathcal{N}}_{\lambda} \,\right)}_{z}.  
 \end{equation*}
 It follows from the definitions $g_{ {}_{\lambda} } = \frac{g}{\, 1 + i \lambda g\,}$, $h=\frac{1}{\,g\,}$, and $h_{{}_{\lambda}}= h+ i \lambda$ that
    \begin{equation*}  \label{FI01q06}
 \frac{1}{\, g_{{}_{\lambda}}}  =  \frac{1}{\, g \,} + i \lambda = h + i \lambda = h_{{}_{\lambda}}. 
 \end{equation*}
We set
  \begin{equation*}  \label{FI01q07}
    {\mathcal{P}}_{\lambda} := {\mathcal{P}} = \frac{1}{\,2\,}  {\mathcal{N}} =\frac{1}{\,2\,} {\mathcal{N}}_{\lambda}.
 \end{equation*}
   Theorem \ref{EBmain} shows that the triple $\left( g_{{}_{\lambda}}, \, {\mathcal{P}}_{\lambda}, \, {\mathcal{Q}}_{\lambda} \right)$ is a Weierstrass data of the second kind.
It now remains to deduce the integrability condition: 
 \begin{eqnarray*}  \label{FI01q08}
      {\left( \, {\mathcal{Q}}_{\lambda} \, \right)}_{z}
    &=&   h_{{}_{\lambda}}  \,  {\left(\, {\mathcal{M}}_{\lambda} \,\right)}_{z} - \frac{\, {h_{{}_{\lambda}} }^2 \,}{2} {\left(\, {\mathcal{N}}_{\lambda} \,\right)}_{z}    
  =    \frac{1}{\, g_{{}_{\lambda}}}  \,  {\mathcal{M}}_{z} -   \frac{1}{\, { g_{{}_{\lambda}}}^{2} }  {\mathcal{P}}_{z}  \\   
 &=&    \frac{1}{\, g_{{}_{\lambda}}} \left( \, 
 \frac{1}{\, g \,} \,  {\mathcal{P}}_{z} + g \, {\mathcal{Q}}_{z} -   \frac{1}{\, g_{{}_{\lambda}}}   {\mathcal{P}}_{z}
 \, \right)   
=     \frac{1}{\, g_{{}_{\lambda}}} \left( \,  g \, {\mathcal{Q}}_{z}  +  \left(\, \frac{1}{\, g \,} -   \frac{1}{\, g_{{}_{\lambda}}} \, \right) {\mathcal{P}}_{z}
 \, \right)   \\  
  &=&  \left( \,  \frac{1}{\,g\,} + i \lambda \, \right) \left(  g    {\mathcal{Q}}_{z} - i \lambda {\mathcal{P}}_{z}      \right).
 \end{eqnarray*} 
\end{proof}

\begin{prop}[\textbf{Elliptic Deformations}] \label{FIS02}
 Let the triple $\left(h, \, \mathcal{M}, \, \mathcal{N}\right)$ be a Weierstrass data of the second kind. Given a parameter $\tau \in \mathbb{R}$,
 we set $h_{ {}_{\tau} } := e^{-i \tau} h$ and ${\mathcal{N}}_{\tau} :=\mathcal{N}$. 
 Then, there exists a ${\mathcal{C}}^{2}$ function ${\mathcal{M}}_{\tau}: \Omega \to \mathbb{R}$ satisfying the integrability condition 
 \begin{equation*}  \label{FI02}
      {\left( \, {\mathcal{M}}_{\tau} \, \right)}_{z} = e^{i \tau}  \, {\mathcal{M}}_{ z } - \frac{\,  e^{i \tau} -  e^{- i \tau}   \,}{2}  \, h \,  {\mathcal{N}}_{ z }. 
 \end{equation*}
 Moreover, the triple $\left(h_{ {}_{\tau} }, \,{\mathcal{M}}_{\tau}, \, {\mathcal{N}}_{\tau}\right)$ becomes a Weierstrass data of the second kind. 
\end{prop}

\begin{proof}[\textbf{First Proof of Proposition \ref{FIS02}}]  Since $h=h_{ {}_{0} }$ vanishes nowhere, $h_{ {}_{\tau}}=e^{-i \tau} h$ also vanishes nowhere. 
We introduce the function $  \mathcal{E}: \Omega \to \mathbb{R}$ by 
 \begin{equation*}  \label{FI02p03}
  \mathcal{E}:=e^{i \tau}  \, {\mathcal{M}}_{ z } - \frac{\,  e^{i \tau} -  e^{- i \tau}   \,}{2}  \, h \,  {\mathcal{N}}_{ z }.
 \end{equation*}
Note that $h$ is holomorphic and that $   {\mathcal{M}}_{ z \overline{\,z\,} }= \frac{\, h+ \overline{\, h \,} \,}{2} \,  {\mathcal{N}}_{ z \overline{\,z\,} }$. We have
         \begin{eqnarray*}  \label{FI02p04}
      {\mathcal{E}}_{ \overline{\,z\,} } &=&  e^{i \tau}  \, {\mathcal{M}}_{ z \overline{\,z\,} } - \frac{\,  e^{i \tau} -  e^{- i \tau}   \,}{2}  \, h \,  {\mathcal{N}}_{ z \overline{\,z\,}} 
          =  \left( \, e^{i \tau}  \,  \frac{\, h+ \overline{\, h \,} \,}{2}  - \frac{\,  e^{i \tau} -  e^{- i \tau}   \,}{2}  \, h \,  \right)  {\mathcal{N}}_{ z \overline{\,z\,}}   \\
       &=&   \left( \,   \frac{ \,  e^{-i \tau} h +  e^{i \tau} \overline{\, h \,}     \, }{2} \, \right) {\mathcal{N}}_{ z \overline{\,z\,}} =  \left( \textrm{Re} \,   h_{ {}_{\tau}}  \right) {\mathcal{N}}_{ z \overline{\,z\,}}.
 \end{eqnarray*}
Since the domain $\Omega$ is simply connected, by Poincar\'{e}'s Lemma, the observation 
 \begin{equation*}  \label{FI02p05}
  {\mathcal{E}}_{ \overline{\,z\,} } =  \left( \textrm{Re} \,   h_{ {}_{\tau}}  \right) {\mathcal{N}}_{ z \overline{\,z\,}} \in \mathbb{R}
   \end{equation*} 
    implies the existence of an 
 ${\mathbb{R}}$-valued function ${\mathcal{M}}_{\lambda}$ defined on $\Omega$ such that
 \begin{equation*}  \label{FI02p06}
        {\left( \, {\mathcal{M}}_{\lambda} \, \right)}_{z}  =  \mathcal{E}.
 \end{equation*}           
It follows that  
 \begin{equation*}  \label{FI02p07}
          {\left( \, {\mathcal{M}}_{\lambda} \, \right)}_{z  \overline{\,z\,} }   =  {\mathcal{E}}_{\overline{\,z\,} }
          = \left( \textrm{Re} \,   h_{ {}_{\tau}}  \right) {\mathcal{N}}_{ z \overline{\,z\,}}
          =  \left( \textrm{Re} \,   h_{ {}_{\tau}}  \right)   {\left( \, {\mathcal{N}}_{\lambda} \, \right)}_{ z \overline{\,z\,}}.
 \end{equation*}    
 Finally, we deduce the last condition
 \begin{eqnarray*}  \label{FI02p08}
    {\left( \, {\mathcal{M}}_{\lambda} \, \right)}_{z}  -  \left( \textrm{Re} \,   h_{ {}_{\tau}}  \right)   {\left( \, {\mathcal{N}}_{\lambda} \, \right)}_{ z }
    &=&      {\mathcal{E}} - \left( \,   \frac{ \,   e^{-i \tau} h +  e^{i \tau} \overline{\, h \,}    \, }{2} \, \right)  {\mathcal{N}}_{ z }  \\
    &=&  e^{i \tau}  \, {\mathcal{M}}_{ z } - \frac{\,  e^{i \tau} -  e^{- i \tau}   \,}{2}  \, h \,  {\mathcal{N}}_{ z } -  \left( \,   \frac{ \,   e^{-i \tau} h +  e^{i \tau} \overline{\, h \,}     \, }{2} \, \right)  {\mathcal{N}}_{ z }  \\
    &=&  e^{i \tau} \left( \,   { \mathcal{M}}_{z} - \left( \textrm{Re} \,  h \right)  \,  { \mathcal{N}}_{z}  \, \right) \neq 0.
 \end{eqnarray*}
\end{proof}

Though the first proof of Proposition \ref{FIS02} is self-contained, it does not explain how to discover the integrability condition in Proposition \ref{FIS02}:
 \begin{equation*}  \label{FI02p01again}
      {\left( \, {\mathcal{M}}_{\tau} \, \right)}_{z} = e^{i \tau}  \, {\mathcal{M}}_{ z } - \frac{\,  e^{i \tau} -  e^{- i \tau}   \,}{2}  \, h \,  {\mathcal{N}}_{ z }. 
 \end{equation*}
We use the equivalence of the Weierstrass data of the first kind and the Weierstrass data of the second kind to discover the integrability condition. 

\begin{proof}[\textbf{Second Proof of Proposition \ref{FIS02}}] 
We set $g=\frac{1}{\,h\,}$ and $\mathcal{P}=\frac{1}{\,2\,} \mathcal{N}$. According to Theorem \ref{EBmain}, we can
 find a ${\mathcal{C}}^{2}$ function ${\mathcal{Q}}: \Omega \to \mathbb{R}$ such that 
 \begin{equation*}  \label{FI02q01}
        {\mathcal{Q}}_{z} = h \,  {\mathcal{M}}_{z} - \frac{\, h^2 \,}{2} {\mathcal{N}}_{z}.  
 \end{equation*}
 Moreover, the triple $\left(g, \, \mathcal{P}, \, \mathcal{Q}\right)$ is a Weierstrass data of the second kind: 
  \begin{equation*} \label{FI02q02}
    g_{\overline{\,z\,} } =0, \quad
    {\mathcal{P}}_{z \overline{\,z\,}} =  {\vert \, g \, \vert}^{2}  \;  {\mathcal{Q}}_{z \overline{\,z\,}}, \quad 
    {\mathcal{P}}_{z}  -  {\vert \, g \, \vert}^{2} \,  {\mathcal{Q}}_{z} \neq 0.  
 \end{equation*}
We observe that $\vert \, e^{i \tau} g \, \vert = \vert \, g \, \vert$. We find immediately that the triple 
\begin{equation*} \label{FI02q03}
  \left({g}_{{}_{\tau}}, \, {\mathcal{P}}_{\tau}, \, {\mathcal{Q}}_{\tau}\right)=\left(e^{i \tau} g, \, \mathcal{P}, \, \mathcal{Q}\right)
 \end{equation*}
is also a Weierstrass data of the first kind:
  \begin{equation*} \label{FI02q04}
    {\left(\, {g}_{{}_{\tau}} \,\right)}_{\overline{\,z\,} } =0, \quad
    { \left( {\mathcal{P}}_{\tau} \right) }_{z \overline{\,z\,}} =  {\vert \, g \, \vert}^{2}  \;  { \left( {\mathcal{Q}}_{\tau}  \right)}_{z \overline{\,z\,}}, \quad 
   {  \left( {\mathcal{P}}_{\tau} \right) }_{z}  -  {\vert \, {g}_{{}_{\tau}} \, \vert}^{2} \,  {\left( {\mathcal{Q}}_{\tau} \right)}_{z} \neq 0.  
 \end{equation*}
   Theorem \ref{EAmain} guarantees the existence of a ${\mathcal{C}}^{2}$ function ${\mathcal{M}}_{\tau}: \Omega \to \mathbb{R}$ such that 
 \begin{equation*}  \label{FI02q05}
  {\left( {\mathcal{M}}_{\tau} \right)}_{z} = \frac{1}{\, {g}_{{}_{\tau}} \,} \,  {\left( {\mathcal{P}}_{\tau} \right)}_{z} + {g}_{{}_{\tau}} \, {\left( {\mathcal{Q}}_{\tau} \right)}_{z}
 \end{equation*}
 It follows from the definitions ${h}_{{}_{\tau}}  = e^{-i \tau} h $, $g=\frac{1}{\,h\,}$, and ${g}_{{}_{\tau}}=e^{i \tau} g$  that
    \begin{equation*}  \label{FI02q06}
 {h}_{{}_{\tau}}  = \frac{1}{\, {g}_{{}_{\tau}} \,}.  
 \end{equation*}
We set
  \begin{equation*}  \label{FI02q07}
    {\mathcal{N}}_{\tau}:= {\mathcal{N}} =2 {\mathcal{P}} =2 {\mathcal{P}}_{\lambda}.
 \end{equation*}
   Theorem \ref{EAmain} shows that the triple $\left( h_{{}_{\tau}}, \, {\mathcal{M}}_{\tau}, \, {\mathcal{N}}_{\tau} \right)$ is a Weierstrass data of the second kind.
It now remains to deduce the integrability condition:
 \begin{eqnarray*}  \label{FI02q08}
    {\left( \, {\mathcal{M}}_{\tau} \, \right)}_{z} 
    &=&  \frac{1}{\, {g}_{{}_{\tau}} \,} \,  {\left( {\mathcal{P}}_{\tau} \right)}_{z} + {g}_{{}_{\tau}} \, {\left( {\mathcal{Q}}_{\tau} \right)}_{z} 
    = e^{-i \tau} h \frac{\,  {\mathcal{N}}_{z} \, }{2} + \frac{e^{i \tau}\,}{h}    {\mathcal{Q}}_{z}   \\
         &=&   \frac{\, e^{-i \tau}  h \, }{2} {\mathcal{N}}_{z} + \frac{e^{i \tau}\,}{h}   \left( \,    h \,  {\mathcal{M}}_{z} - \frac{\, h^2 \,}{2} {\mathcal{N}}_{z}   \, \right) 
 = e^{i \tau}  \, {\mathcal{M}}_{ z } - \frac{\,  e^{i \tau} -  e^{- i \tau}   \,}{2}  \, h \,  {\mathcal{N}}_{ z }.
 \end{eqnarray*} 
\end{proof}

 \begin{prop}[\textbf{Hyperbolic Deformations in terms of the Weierstrass data of the first kind}] \label{FIS03A}
 Let the triple $\left(g, \, \mathcal{P}, \, \mathcal{Q}\right)$ be a Weierstrass data of the first kind. Let $\eta \in \mathbb{R}$ be a parameter. 
 Then, the triple $\left({g}_{ {}_{\eta} }, \, {\mathcal{P}}_{\eta}, \, {\mathcal{Q}}_{\eta}\right)
 :=\left( e^{\eta} g, \, e^{\eta} {\mathcal{P}}, \, e^{-\eta} \mathcal{Q}  \right)$ becomes a Weierstrass data of the first kind.
 \end{prop}

\begin{proof}  The assumption says that the triple $\left(g, \, \mathcal{P}, \, \mathcal{Q}\right)$ satisfy
     \begin{equation*}   \label{FI0Aq01}
     g_{\overline{\,z\,} } =0, \quad
    {\mathcal{P}}_{z \overline{\,z\,}} =  {\vert \, g \, \vert}^{2}  \;  {\mathcal{Q}}_{z \overline{\,z\,}}, \quad 
         {\mathcal{P}}_{z} -  {\vert \, g \, \vert}^{2}   {\mathcal{Q}}_{z} \neq 0.
 \end{equation*}
Hence, the triple  $\left({g}_{ {}_{\eta} }, \, {\mathcal{P}}_{\eta}, \, {\mathcal{Q}}_{\eta}\right)$ satisfy
     \begin{eqnarray*}   \label{FI03Aq02}
 &&   {\left(\, {g}_{ {}_{\eta} } \,\right)}_{\overline{\,z\,} } =  {\left(\,   e^{\eta} g   \,\right)}_{\overline{\,z\,} } =0, \\
 &&  {\left(\, {\mathcal{P}}_{\eta} \, \right)}_{ z \overline{\,z\,} }= {\left(\,  e^{\eta} {\mathcal{P}}    \, \right)}_{ z \overline{\,z\,} } 
= e^{\eta}   {\vert \, g \, \vert}^{2}  \;  {\mathcal{Q}}_{z \overline{\,z\,}} = {\vert \,   e^{\eta}  g \, \vert}^{2}  \,  e^{-\eta}   {\mathcal{Q}}_{z \overline{\,z\,}} 
 =  {\vert \, {g}_{ {}_{\eta} } \vert}^{2}  \;  {\left( \, {\mathcal{Q}}_{\eta} \, \right)}_{z \overline{\,z\,}}  \\
 &&    {\left(\, {\mathcal{P}}_{\eta} \, \right)}_{z} -  {\vert \,  {g}_{ {}_{\eta} } \vert}^{2}   {\left(\, {\mathcal{Q}}_{\eta} \, \right)}_{z} 
 = e^{\eta} \left( \, {\mathcal{P}}_{z} -  {\vert \, g \, \vert}^{2}   {\mathcal{Q}}_{z} \, \right)  \neq 0. 
 \end{eqnarray*}
\end{proof}

 \section{Reduction of Liu integrable system} \label{Liu system}
  
The purpose of this section is to explain one way to discover our integrable systems for marginally trapped surfaces in ${\mathbb{L}}^{4}$. An algebraic reduction of the following version of the Liu 
integrable system \cite{Liu 2013} for the triple of a holomorphic function and two $\mathbb{C}$-valued functions yields to our integrable systems for the Weierstrass
triple of a holomorphic function and two $\mathbb{R}$-valued functions. We emphasize that the hidden geometric idea of our algebraic reduction is implicitly contained in our new Weierstrass representations (Theorem \ref{Poisson 01}, Corollary  \ref{Poisson 02}, and Corollary  \ref{Poisson 03}) for marginally trapped surfaces.
  
    \begin{lem} [\textbf{Liu's integrable system for marginally trapped surfaces in ${\mathbb{L}}^{4}$ \cite{Liu 2013}}] \label{Liu}  
    Let $\Omega$ be an open domain in $\mathbb{C}$ with the complex coordinate $z$.  
We assume that a spacelike surface $\Sigma$ in Lorentz-Minkowski space ${\mathbb{L}}^{4}$ is parameterized by a ${\mathcal{C}}^{2}$ conformal patch $\mathbf{X}:  \Omega \to {\mathbb{L}}^{4}$ of the form 
 \begin{equation*} \label{Liu01} 
 {\mathbf{X}}_{z}   =  \begin{bmatrix}  \; {\left({\mathbf{x}}_{1} \right)}_{z} \; \\[4pt]   {\left({\mathbf{x}}_{2} \right)}_{z} \\[4pt]  {\left({\mathbf{x}}_{3} \right)}_{z} \\[4pt]  {\left({\mathbf{x}}_{4} \right)}_{z} \end{bmatrix}  
 =  \Psi \begin{bmatrix}   f_{1} + f_{2}  \\[4pt]  \; -i \left(  f_{1} - f_{2}  \right) \; \\[4pt]   1 - f_{1} f_{2}  \\[4pt]   1 + f_{1} f_{2}  \end{bmatrix},
  \end{equation*}
where  $\Psi :  \Omega \to \mathbb{C}-\{ 0 \}$ and $f_{1}, f_{2}  :  \Omega \to \mathbb{C}$. Then, the following three integrability conditions hold for all $z \in \Omega$:
 \begin{enumerate}
 \item[\textbf{(1)}] ${\Psi}_{\overline{\, z \,}} = \overline{\,  {\Psi}_{\overline{\, z \,}}   \,}$,
  \item[\textbf{(2)}] ${\left(\,\Psi f_{1} f_{2} \, \right)}_{\overline{\, z \,}} = \overline{\, {\left(\,\Psi f_{1} f_{2} \, \right)}_{\overline{\, z \,}}  \,}$,
 \item[\textbf{(3)}] ${\left(\, \Psi f_{1} \,\right)}_{\overline{\, z \,}} = \overline{\, {\left(\, \Psi f_{2} \,\right)}_{\overline{\, z \,}}    \,} \quad \Leftrightarrow \quad 
 \overline{\, {\left(\, \Psi f_{1} \,\right)}_{\overline{\, z \,}}    \,}= {\left(\, \Psi f_{2} \,\right)}_{\overline{\, z \,}}$.
 \end{enumerate} 
If $\Sigma$ is marginally trapped, then the following holomorphicity condition holds for each point $z \in \Omega$:
 \begin{equation*} \label{Liu02} 
     {\left(\, f_{1} \, \right)}_{\overline{\, z \,}}  = 0 \quad \text{or}  \quad  {\left(\, f_{2} \, \right)}_{\overline{\, z \,}}  = 0, 
 \end{equation*}
 \end{lem}
     

\begin{proof} It follows from \cite[Theorem 4.1 and Theorem 4.2]{Liu 2013}. For the reader's convenience, we sketch a self-contained proof. 
We first observe that 
 \begin{eqnarray*} \label{LiuP01} 
        \Psi &=& \frac{ \,{\left({\mathbf{x}}_{3} \right)}_{z}  + {\left({\mathbf{x}}_{4} \right)}_{z}\, }{2}, \\ 
    \Psi f_{1} f_{2} &=& \frac{ \, - {\left({\mathbf{x}}_{3} \right)}_{z}  +  {\left({\mathbf{x}}_{4} \right)}_{z} \, }{2}, \\   
       \Psi f_{1} &=& \frac{ \, {\left({\mathbf{x}}_{1} \right)}_{z}  + i {\left({\mathbf{x}}_{2} \right)}_{z} \, }{2}, \\  
    \Psi f_{2} &=& \frac{ \,  {\left({\mathbf{x}}_{1} \right)}_{z}  - i  {\left({\mathbf{x}}_{2} \right)}_{z} \, }{2}. \\       
 \end{eqnarray*}
 (These relations reveal how to prescribe the three functions $\Psi$, $f_{1}$, $f_{2}$ in terms of the four component functions of ${\mathbf{X}}_{z}$.) 
Since $ {\mathbf{x}}_{1}=\overline{\, {\mathbf{x}}_{1} \,}$,    $ {\mathbf{x}}_{2}=\overline{\, {\mathbf{x}}_{2} \,}$,    $ {\mathbf{x}}_{3}=\overline{\,
{\mathbf{x}}_{3} \,}$,    $ {\mathbf{x}}_{4}=\overline{\, {\mathbf{x}}_{4} \,}$, we have 
 \begin{equation*} \label{LiuP02} 
    \overline{\,  {\Psi}_{\overline{\, z \,}}   \,} 
   =  \overline{\,   {\left( \, \frac{ \,{\left({\mathbf{x}}_{3} \right)}_{z \overline{\, z \,}} +  {\left({\mathbf{x}}_{4} \right)}_{z \overline{\, z \,}}\, }{2}  \, \right) }   \,} 
    =  {\left( \, \frac{ \, {\left({\mathbf{x}}_{3} \right)}_{z} +{\left({\mathbf{x}}_{4} \right)}_{z}\, }{2}  \, \right) }_{\overline{\, z \,}} =    {\Psi}_{\overline{\, z \,}} 
 \end{equation*}
and 
 \begin{equation*} \label{LiuP03} 
    \overline{\,  {\left(\,\Psi f_{1} f_{2} \, \right)}_{\overline{\, z \,}}   \,} 
   =  \overline{\,   {\left( \, \frac{ \, - {\left({\mathbf{x}}_{3} \right)}_{z \overline{\, z \,}} +  {\left({\mathbf{x}}_{4} \right)}_{z \overline{\, z \,}}\, }{2}  \, \right) }   \,} 
    =  {\left( \, \frac{ \, - {\left({\mathbf{x}}_{3} \right)}_{z} +{\left({\mathbf{x}}_{4} \right)}_{z}\, }{2}  \, \right) }_{\overline{\, z \,}} =    {\left(\,\Psi f_{1} f_{2} \, \right)}_{\overline{\, z \,}}. 
 \end{equation*}
 We also have 
  \begin{eqnarray*}  \label{LiuP04} 
   \overline{\, {\left(\, \Psi f_{2} \,\right)}_{\overline{\, z \,}}    \,}  &=&    
  \overline{\,  {\left( \,    \frac{ \, {\left({\mathbf{x}}_{1} \right)}_{z}  - i {\left({\mathbf{x}}_{2} \right)}_{z} \, }{2}     \, \right)}_{ \overline{\, z \,} }   \,}
  = {\left( \,   \frac{ \, {\left({\mathbf{x}}_{1} \right)}_{\overline{\, z \,}}  + i {\left({\mathbf{x}}_{2} \right)}_{\overline{\, z \,}} \, }{2}  \, \right)}_{ z }  \\
  &=&   {\left( \,   \frac{ \, {\left({\mathbf{x}}_{1} \right)}_{z}  + i {\left({\mathbf{x}}_{2} \right)}_{z} \, }{2}  \, \right)}_{ \overline{\, z \,} } 
 =  {\left(\, \Psi f_{1} \,\right)}_{\overline{\, z \,}}.   
 \end{eqnarray*}
 Now, we assume that the spacelike surface $\Sigma$ with the mean curvature vector $\mathbf{H}$ is marginally trapped. Since $\mathbf{X}$ is a conformal patch, we find that
 the null condition
  \begin{equation*} \label{LiuP05} 
  \langle \mathbf{H}, \,  \mathbf{H} \rangle = 0
 \end{equation*}
is equivalent to 
   \begin{equation*} \label{LiuP06} 
  \langle  {\mathbf{X}}_{z \overline{\, z \,}}, \, {\mathbf{X}}_{z \overline{\, z \,}} \rangle = 0.
 \end{equation*}
We differentiate both sides of 
 \begin{equation*} \label{LiuP07} 
 {\mathbf{X}}_{z}   =  \Psi {\mathbf{F}}_{0}, \quad \text{where} \quad   {\mathbf{F}}_{0} :=\begin{bmatrix}   f_{1} + f_{2}  \\ \; -i \left(  f_{1} - f_{2}  \right) \; \\  1 - f_{1} f_{2}  \\  1 + f_{1} f_{2}  \end{bmatrix},  
  \end{equation*}
  with respect to $\overline{\, z \,}$ to obtain the decomposition
 \begin{equation*} \label{LiuP08} 
 {\mathbf{X}}_{z \overline{\, z \,} }   =  {\Psi}_{ \overline{\, z \,} } {\mathbf{F}}_{0} + \Psi    {\left(\, f_{1} \, \right)}_{\overline{\, z \,}}   {\mathbf{F}}_{2} 
 +  \Psi    {\left(\, f_{2} \, \right)}_{\overline{\, z \,}}   {\mathbf{F}}_{1}, 
  \end{equation*}
  where we define
  \begin{equation*} \label{LiuP09} 
     {\mathbf{F}}_{2} :=  \begin{bmatrix}  1 \\ \; -i  \; \\ - f_{2}  \\ f_{2}  \end{bmatrix}  
     \quad \text{and} \quad 
     {\mathbf{F}}_{1} :=   \begin{bmatrix}   1  \\ i  \; \\  - f_{1}   \\ f_{1}   \end{bmatrix}.  
  \end{equation*}
A straightforward computation shows that 
  \begin{equation*} \label{LiuP10} 
     \begin{cases}
          \langle     {\mathbf{F}}_{1}, \,     {\mathbf{F}}_{2} \rangle = 2, \\
          \langle     {\mathbf{F}}_{i}, \,     {\mathbf{F}}_{j} \rangle = 0 \; \text{for} \; \left(i, j \right) \neq \left(1, 2 \right),  \left(2, 1 \right). 
     \end{cases}
  \end{equation*}
It follows that
    \begin{equation*} \label{LiuP11} 
  0 = \langle  {\mathbf{X}}_{z \overline{\, z \,}}, \, {\mathbf{X}}_{z \overline{\, z \,}} \rangle = 2  {\Psi}^{2}    {\left(\, f_{1} \, \right)}_{\overline{\, z \,}}  
   {\left(\, f_{2} \, \right)}_{\overline{\, z \,}}  \,  \langle     {\mathbf{F}}_{1}, \,     {\mathbf{F}}_{2} \rangle  = 4  {\Psi}^{2}    {\left(\, f_{1} \, \right)}_{\overline{\, z \,}}  
   {\left(\, f_{2} \, \right)}_{\overline{\, z \,}}. 
 \end{equation*}
We recall the assumption ${\Psi}(z) \neq 0$ for all $z \in \Omega$.  We conclude that $ {\left(\, f_{1} \, \right)}_{\overline{\, z \,}} =0$ or 
$ {\left(\, f_{2} \, \right)}_{\overline{\, z \,}} =0$ at each point $z \in \Omega$. This completes the proof.
\end{proof}
 
It remains to explain how to use the Liu integrable system for the marginally trapped surface 
\begin{enumerate}
 \item[\textbf{(1)}] ${\Psi}_{\overline{\, z \,}} = \overline{\,  {\Psi}_{\overline{\, z \,}}   \,}$,
  \item[\textbf{(2)}] ${\left(\,\Psi f_{1} f_{2} \, \right)}_{\overline{\, z \,}} = \overline{\, {\left(\,\Psi f_{1} f_{2} \, \right)}_{\overline{\, z \,}}  \,}$,
 \item[\textbf{(3)}] ${\left(\, \Psi f_{1} \,\right)}_{\overline{\, z \,}} = \overline{\, {\left(\, \Psi f_{2} \,\right)}_{\overline{\, z \,}}    \,} \quad \Leftrightarrow \quad 
 \overline{\, {\left(\, \Psi f_{1} \,\right)}_{\overline{\, z \,}}    \,}= {\left(\, \Psi f_{2} \,\right)}_{\overline{\, z \,}}$,
   \item[\textbf{(4)}]  ${\left(\, f_{1} \, \right)}_{\overline{\, z \,}}  \, {\left(\, f_{2} \, \right)}_{\overline{\, z \,}}  = 0$. 
 \end{enumerate} 
to discover the key equations in our integrable system in Definition \ref{FWD02}:
\begin{enumerate}
 \item[\textbf{(1)}] ${h}_{\overline{\,z\,} } =0$,
  \item[\textbf{(2)}] ${\mathcal{M}}_{ z \overline{\,z\,} }= \left( \textrm{Re} \,  h \right)  \,  {\mathcal{N}}_{ z \overline{\,z\,} }$,
 \end{enumerate} 
 which were essential in our second Weierstrass representation in Corollary \ref{Poisson 02}. 
 We keep the notations in Lemma \ref{Liu} and consider the case ${\left(\, f_{2} \, \right)}_{\overline{\, z \,}} \equiv 0$. We have 
\begin{equation*} \label{fromLiu01} 
       \Psi f_{1} = \frac{ \, {\left({\mathbf{x}}_{1} \right)}_{z}  + i {\left({\mathbf{x}}_{2} \right)}_{z} \, }{2} \quad \text{and} \quad
    \Psi f_{2} = \frac{ \,  {\left({\mathbf{x}}_{1} \right)}_{z}  - i  {\left({\mathbf{x}}_{2} \right)}_{z} \, }{2}. \\       
\end{equation*}
  We make an additional assumption that 
the function $\Psi = \frac{ \,{\left({\mathbf{x}}_{3} \right)}_{z}  + {\left({\mathbf{x}}_{4} \right)}_{z}\, }{2}$ never vanish on $\Omega$. (This assumption will be not required in 
our second Weierstrass representation in Corollary \ref{Poisson 02}.) We define the three functions  
${\mathcal{M}}, {\mathcal{N}} :  \Omega \to \mathbb{R}$ and  $h: \Omega \to \mathbb{C}$ by 
\begin{equation*} \label{fromLiu02} 
          \quad  {\mathcal{M}} :=  {\mathbf{x}}_{1}, \quad  {\mathcal{N}} :=   {\mathbf{x}}_{3} + {\mathbf{x}}_{4}, \quad
               h := f_{2}  = \frac{\, \Psi   f_{2} \,}{\Psi  }  = \frac{\, {\left({\mathbf{x}}_{1} \right)}_{z} -
                i {\left({\mathbf{x}}_{2} \right)}_{z}  \,}{\, {\left({\mathbf{x}}_{3} \right)}_{z}  + {\left({\mathbf{x}}_{4} \right)}_{z} \,  } 
               =   \frac{\, {\left({\mathbf{x}}_{1} \right)}_{z} - i {\left({\mathbf{x}}_{2} \right)}_{z}  \,}{\,   {\mathcal{N}}_{z} \,}. 
\end{equation*}
We note that $h=f_{2}$ is holomorphic in $\Omega$ and that ${\mathcal{N}} = \overline{\, {\mathcal{N}}  \,}$. We observe 
 \begin{eqnarray*} \label{fromLiu03} 
2    \Psi f_{2} =  {\left({\mathbf{x}}_{1} \right)}_{z}  - i  {\left({\mathbf{x}}_{2} \right)}_{z} &=&  {\mathcal{N}}_{z} h,  \\
   - i  {\left({\mathbf{x}}_{2} \right)}_{z} &=& - {\left({\mathbf{x}}_{1} \right)}_{z} + \left( \,  {\left({\mathbf{x}}_{1} \right)}_{z} - i {\left({\mathbf{x}}_{2} \right)}_{z}     \, \right)
   = - {\mathcal{M}}_{z} +   {\mathcal{N}}_{z}  h, \\  
2     \Psi f_{1} = {\left({\mathbf{x}}_{1} \right)}_{z}  + i  {\left({\mathbf{x}}_{2} \right)}_{z} &=&  {\mathcal{M}}_{z}  - \left(  - {\mathcal{M}}_{z} +   {\mathcal{N}}_{z}  h \right) 
   =  2 {\mathcal{M}}_{z} -   {\mathcal{N}}_{z}  h. 
 \end{eqnarray*}
We conclude that 
 \begin{eqnarray*} \label{fromLiu04} 
 0 &=&  {\left(\,  \Psi f_{1} \,\right)}_{\overline{\, z \,}} - \overline{\, {\left(\,  \Psi f_{2} \,\right)}_{\overline{\, z \,}}    \,}  \\
    &=& {\left(\,    {\mathcal{M}}_{z} -  \frac{1}{\, 2 \, } {\mathcal{N}}_{z}  h       \,\right)}_{\overline{\, z \,}} - \overline{\, {\left(\,    \frac{1}{\, 2 \, } {\mathcal{N}}_{z}  h  \,\right)}_{\overline{\, z \,}}    \,}  \\ 
    &=& \left(\,    {\mathcal{M}}_{z  \overline{\, z \,} } -  \frac{1}{\, 2 \, } {\mathcal{N}}_{z\overline{\, z \,}}  h   - 
    \frac{1}{\, 2 \, } {\mathcal{N}}_{z}  h_{\overline{\, z \,}}   \,\right) -   \overline{\, {\left(\, \frac{1}{\, 2 \, } {\mathcal{N}}_{z\overline{\, z \,}}  h +
    \frac{1}{\, 2 \, } {\mathcal{N}}_{z}  h_{\overline{\, z \,}}  \,\right)}      \,}  \\ 
    &=& {\mathcal{M}}_{ z \overline{\,z\,} } - \left( \textrm{Re} \,  h \right)  \,  {\mathcal{N}}_{ z \overline{\,z\,} }  -  \textrm{Re} \,  \left(  h_{ \overline{\, z \,} } {\mathcal{N}}_{z} \right) \\
   &=& {\mathcal{M}}_{ z \overline{\,z\,} } - \left( \textrm{Re} \,  h \right)  \,  {\mathcal{N}}_{ z \overline{\,z\,} }.  
 \end{eqnarray*}
This completes the desired reduction. 

\begin{rem} \label{geo alg}
 The geometric idea to construct our algebraic reduction of the Liu integrable system is revealed our Weierstrass represention (Corollary \ref{Poisson 02} and Remark \ref{Poisson 02 Remarks}) for 
 marginally trapped surfaces in ${\mathbb{L}}^{4}$, which generalizes Weierstrass representation of the second kind (due to O. Kobayashi \cite[Corollary 1.3]{Kobayashi 1983}) for maximal surfaces in ${\mathbb{L}}^{3}$.
\end{rem}

 \section{Three Weierstrass representations} \label{three Weierstrass}

The $4$-dimensional Lorentz-Minkowski space ${\mathbb{L}}^{4}$ is the real vector space ${\mathbb{R}}^{4}$ equipped with the Lorentzian metric
 \begin{equation*} \label{LM 01} 
  \langle \; , \; \rangle = {dx_{1}}^{2}+ {dx_{2}}^{2}+{dx_{3}}^{2}-{dx_{4}}^{2}, 
 \end{equation*}
 where $x_{1}, x_{2}, x_{3}, x_{4}$ denotes the canonical coordinates in ${\mathbb{R}}^{4}$. The standard complexificiation of $\langle \; , \; \rangle$ induces 
 the symmetric bilinear form on ${\mathbb{C}}^{4}$. 

  The purpose of this section is to present various conformal representation for marginally trapped surfaces in ${\mathbb{L}}^{4}$. A spacelike surface in Lorentz-Minkowski space ${\mathbb{L}}^{4}$ is called a marginally trapped surface if its mean curvature vector $\mathbf{H}$ satisfies the null condition:
  \begin{equation*} 
      \langle \mathbf{H},  \mathbf{H} \rangle=0.
  \end{equation*}

  \begin{thm} [\textbf{First Weierstrass representation for marginally trapped surfaces in ${\mathbb{L}}^{4}$}] \label{Poisson 01} 
Let $\Omega \subset {\mathbb{R}}^2 \equiv \mathbb{C}$ be a simply connected domain with the complex coordinate $z$.
Let the triple $\left(g, \, \mathcal{P}, \, \mathcal{Q}\right)$ be a Weierstrass data of the first kind.
The holomorphic function $g : \Omega \to \mathbb{C}-\left\{0\right\}$ and two ${\mathcal{C}}^{2}$ functions 
$\mathcal{P},  \mathcal{Q}: \Omega \to \mathbb{R}$ satisfy the equation
 \begin{equation*} \label{Poisson 01a} 
 {\mathcal{P}}_{z \overline{\,z\,}} =  {\vert \, g \, \vert}^{2}  \;  {\mathcal{Q}}_{z \overline{\,z\,}},
 \end{equation*}
 and the condition
  \begin{equation*} \label{Poisson 01b} 
  {\mathcal{P}}_{z}   \neq  {\vert \, g \, \vert}^{2} \,  {\mathcal{Q}}_{z}.
 \end{equation*}
Then, there exists a conformal parameterization $\mathbf{X}={\mathbf{X}}_{\left[\,g, \, \mathcal{P}, \, \mathcal{Q} \, \right]}:  \Omega \to {\mathbb{L}}^{4}$ of the marginally trapped surface 
$\Sigma=\mathbf{X} \left( \Omega \right)$ in Lorentz-Minkowski space ${\mathbb{L}}^{4}$ satisfying the following two equalities:
 \begin{equation*} \label{Poisson 01c} 
 {\mathbf{X}}_{z}   =   \begin{bmatrix}  \; {\left({\mathbf{x}}_{1} \right)}_{z} \; \\[4pt]   {\left({\mathbf{x}}_{2} \right)}_{z} \\[4pt]  {\left({\mathbf{x}}_{3} \right)}_{z} \\[4pt]  {\left({\mathbf{x}}_{4} \right)}_{z} \end{bmatrix}  
 =  {\mathcal{P}}_{z} \begin{bmatrix}  \frac{1}{\,g\,} \\[4pt]  \; \frac{i}{\,g\,}  \; \\[4pt]  1 \\[4pt]  1 \end{bmatrix} 
  +  {\mathcal{Q}}_{z} \begin{bmatrix} g \\[4pt]  - i g \\[4pt]   -1  \\[4pt]   1 \end{bmatrix}  \quad \text{and} \quad
 {\mathbf{X}}_{z \overline{\,z\,}} = {\mathcal{Q}}_{z \overline{\,z\,}}
 \begin{bmatrix} 2 \, \textrm{Re} \, g  \\[4pt]   2 \, \textrm{Im} \, g \\[4pt]   -1+{\vert  \, g \, \vert}^{2}  \\[4pt]   1+{\vert  \, g \, \vert}^{2} \end{bmatrix}.
  \end{equation*}
  In particular, we have ${\mathbf{x}}_{3}(z)=\mathcal{P}(z)-\mathcal{Q}(z)$ and ${\mathbf{x}}_{4}(z)=\mathcal{P}(z)+\mathcal{Q}(z)$, up to additive constants.
The Gauss map ${\mathcal{N}}$ of the marginally trapped surface $\Sigma$ defined by 
   \begin{equation*}\label{Poisson 01d} 
   {\mathcal{N}}:= \begin{bmatrix} 2 \, \textrm{Re} \, g  \\[4pt]   2 \, \textrm{Im} \, g \\[4pt]   -1+{\vert  \, g \, \vert}^{2}  \\[4pt]   1+{\vert  \, g \, \vert}^{2} \end{bmatrix}
  \end{equation*}
 satisfies the orthogonality conditions $\langle  {\mathbf{X}}_{z} , {\mathcal{N}} \rangle=0$ and $\langle {\mathcal{N}}, {\mathcal{N}} \rangle=0$. (In particular, $\mathcal{N}$ is 
null.) The conformal metric induced by
 the patch ${\mathbf{X}}(z)$ is  
     \begin{equation*}  \label{Poisson 01e} 
 {{ds}_{{}_{\Sigma}}}^{2} =
  \frac{4}{ \, {\vert \, g \, \vert}^{2} \, }  \,  {\left\vert \,   {\mathcal{P}}_{z}  -        {\vert \, g \, \vert}^{2}     {\mathcal{Q}}_{z}        \, \right\vert}^{2}  \, {\vert \, dz \, \vert}^{2} .
   \end{equation*}
\end{thm}

\begin{proof}
The vector ${\mathcal{N}}$ is null: 
       \begin{equation*} \label{Poisson 01e17} 
 \langle  {\mathcal{N}} , {\mathcal{N}} \rangle = 
 {\left( \,  2 \, \textrm{Re} \, g \, \right)}^{2} +  {\left( \, 2 \, \textrm{Im} \, g  \, \right)}^{2}  +  {\left( \,  -1+{\vert  \, g \, \vert}^{2} \, \right)}^{2}  -  {\left( \,  1+{\vert  \, g \, \vert}^{2} \, \right)}^{2} 
 =0.
     \end{equation*}    
We introduce the mappings ${\mathcal{E}}, {\mathbf{G}}_{1}, {\mathbf{G}}_{2}$ defined on the domain $\Omega$: 
 \begin{equation*} \label{Poisson 01e02} 
  {\mathbf{G}}_{1}  = \begin{bmatrix}  \frac{1}{\,g\,} \\[4pt]  \; \frac{i}{\,g\,}  \; \\[4pt]  1 \\[4pt]  1 \end{bmatrix},  \;\; 
   {\mathbf{G}}_{2} = \begin{bmatrix} g \\[4pt]  - i g \\[4pt]   -1  \\[4pt]   1 \end{bmatrix},  \;\; 
 {\mathcal{E}}  =  {\mathcal{P}}_{z}  {\mathbf{G}}_{1} +  {\mathcal{Q}}_{z} {\mathbf{G}}_{2}
  \in {\mathbb{C}}^{4}.
  \end{equation*}
  We note that  $ {\mathcal{Q}}_{z \overline{\,z\,}}  \in {\mathbb{R}}$ and ${\mathcal{N}} \in {\mathbb{R}}^{4}$.
We use the holomorphicity of $g$ and the assumption ${\mathcal{P}}_{z  \overline{\,z\,}} =   g  \overline{\, g \,}  \,  {\mathcal{Q}}_{z  \overline{\,z\,}}$ to find that  
the mapping ${\mathcal{E}}_{F}$ is ${\mathbb{R}}^{4}$-valued. Indeed, we have 
 \begin{equation*} \label{Poisson 01e05} 
   {\mathcal{E}}_{\overline{\,z\,}} = {\mathcal{P}}_{z  \overline{\,z\,}} \begin{bmatrix}  \frac{1}{\,g\,} \\[4pt]  \; \frac{i}{\,g\,}  \; \\[4pt]  1 \\[4pt]  1 \end{bmatrix} 
  +  {\mathcal{Q}}_{z  \overline{\,z\,}} \begin{bmatrix} g \\[4pt]  - i g \\[4pt]   -1  \\[4pt]   1 \end{bmatrix}
   = {\mathcal{Q}}_{z \overline{\,z\,}} \begin{bmatrix} 2 \, \textrm{Re} \, g  \\[4pt]   2 \, \textrm{Im} \, g \\[4pt]   -1+{\vert  \, g \, \vert}^{2}  \\[4pt]   1+{\vert  \, g \, \vert}^{2} \end{bmatrix}
     = {\mathcal{Q}}_{z \overline{\,z\,}}  {\mathcal{N}} \in {\mathbb{R}}^{4}.
  \end{equation*}
Since the domain $\Omega$ is simply connected, by Poincar\'{e}'s Lemma, the observation ${\mathcal{E}}_{\overline{\,z\,}} \in {\mathbb{R}}^{4}$ implies the existence of an 
 ${\mathbb{R}}^{4}$-valued mapping $\mathbf{X}$ defined on $\Omega$ such that
 \begin{equation*} \label{Poisson 01e06} 
  \begin{bmatrix}  \; {\left({\mathbf{x}}_{1} \right)}_{z} \; \\[4pt]   {\left({\mathbf{x}}_{2} \right)}_{z} \\[4pt]  {\left({\mathbf{x}}_{3} \right)}_{z} \\[4pt]  {\left({\mathbf{x}}_{4} \right)}_{z} \end{bmatrix} =  {\mathbf{X}}_{z} =  {\mathcal{E}} = {\mathcal{P}}_{z}  {\mathbf{G}}_{1} +  {\mathcal{Q}}_{z} {\mathbf{G}}_{2}  
 =  {\mathcal{P}}_{z} \begin{bmatrix}  \frac{1}{\,g\,} \\[4pt]  \; \frac{i}{\,g\,}  \; \\[4pt]  1 \\[4pt]  1 \end{bmatrix} 
  +  {\mathcal{Q}}_{z} \begin{bmatrix} g \\[4pt]  - i g \\[4pt]   -1  \\[4pt]   1 \end{bmatrix}.   
  \end{equation*}
 It follows from the equalities ${\left({\mathbf{x}}_{3} \right)}_{z} = {\mathcal{P}}_{z} -  {\mathcal{Q}}_{z}$ and ${\left({\mathbf{x}}_{4} \right)}_{z} =  {\mathcal{P}}_{z}+ {\mathcal{Q}}_{z}$ that
  \begin{equation*} \label{Poisson 01e07}  
   {\mathbf{x}}_{3}=\mathcal{P}-\mathcal{Q} \quad \text{and} \quad {\mathbf{x}}_{4}=\mathcal{P}+\mathcal{Q},
    \end{equation*}
   up to additive constants.  
The conformality of the mapping $z \mapsto {\mathbf{X}}(z)$ follows from $   \langle   {\mathbf{X}}_{z} ,  {\mathbf{X}}_{z}  \rangle = 0$. Indeed, we have   
     \begin{equation*} \label{Poisson 01e08} 
     \langle   {\mathbf{X}}_{z} ,  {\mathbf{X}}_{z}  \rangle = 
     {{\mathcal{P}}_{z} }^{2}  \langle   {\mathbf{G}}_{1},   {\mathbf{G}}_{1} \rangle +  
   2  {\mathcal{P}}_{z} {\mathcal{Q}}_{z}  \langle   {\mathbf{G}}_{1},   {\mathbf{G}}_{2} \rangle +
      {\mathcal{Q}}_{z}     \langle   {\mathbf{G}}_{2},   {\mathbf{G}}_{2} \rangle= 0.          
   \end{equation*}  
   It is straightforward to check the identities
   \begin{equation*}  \label{Poisson 01e09}  
       \langle   {\mathbf{G}}_{1},   {\mathbf{G}}_{1} \rangle= {\left( \frac{1}{\,g\,} \right)}^{2} +  {\left( \frac{i}{\,g\,} \right)}^{2} + 1 - 1 =
       0, \;\;    \langle   {\mathbf{G}}_{1},   {\mathbf{G}}_{2} \rangle= 0, \;\;
          \langle   {\mathbf{G}}_{2},   {\mathbf{G}}_{2} \rangle= 0,
   \end{equation*}   
     \begin{equation*}  \label{Poisson 01e10}  
       \langle   {\mathbf{G}}_{1},  \overline{\, {\mathbf{G}}_{1} \,} \rangle= 
       \frac{2}{\, {\vert \, g \, \vert}^{2}\,}, \;\;  
        \langle   {\mathbf{G}}_{1},  \overline{\, {\mathbf{G}}_{2} \,} \rangle=-2, \;\;    \langle  \overline{\, {\mathbf{G}}_{1} \,},    {\mathbf{G}}_{2} \rangle=-2, 
        \;\;      \langle   {\mathbf{G}}_{2},  \overline{\, {\mathbf{G}}_{2} \,} \rangle=2 \, {\vert \, g \, \vert}^{2}. 
   \end{equation*}    
Since $\mathbf{X}$ is  ${\mathbb{R}}^{4}$-valued, we have $ {\mathbf{X}}_{ \overline{\,z\,} }  = \overline{\,  {\mathbf{X}}_{z}   \,}$. To find the conformal metric $
 {{ds}_{{}_{\Sigma}}}^{2} =   {\Lambda(z)} \, {\vert \, dz \, \vert}^{2}$  induced by the patch ${\mathbf{X}}$, we compute the conformal factor ${\Lambda}= 2 \langle   {\mathbf{X}}_{z} ,  {\mathbf{X}}_{ \overline{\,z\,} } \rangle $. We have   
   \begin{eqnarray*}   \label{Poisson 01e11}    
  \frac{1}{\,2\,} {\Lambda(z)} 
  &=&   \langle   {\mathbf{X}}_{z} ,  {\mathbf{X}}_{ \overline{\,z\,} } \rangle =   \langle   {\mathbf{X}}_{z} ,  \overline{\,  {\mathbf{X}}_{z}   \,} \rangle \\
  &=&  \langle  {\mathcal{P}}_{z}  {\mathbf{G}}_{1} +  {\mathcal{Q}}_{z} {\mathbf{G}}_{2}, 
                     \overline{\, {\mathcal{P}}_{z}\,}   \,   \overline{\, {\mathbf{G}}_{1} \, }+     \overline{\,  {\mathcal{Q}}_{z} }  \,   \overline{\, {\mathbf{G}}_{2}\,}   \rangle \\
                    &=&   {\vert \,   {\mathcal{P}}_{z}  \, \vert }^{2}    \,    \langle   {\mathbf{G}}_{1},  \overline{\, {\mathbf{G}}_{1} \,} \rangle
                    +   {\vert \,   {\mathcal{Q}}_{z}  \, \vert }^{2}    \,    \langle   {\mathbf{G}}_{2},  \overline{\, {\mathbf{G}}_{2} \,} \rangle
                    +  {\mathcal{P}}_{z}  \overline{\,   {\mathcal{Q}}_{z}  \,} \, \langle   {\mathbf{G}}_{1},  \overline{\, {\mathbf{G}}_{2} \,} \rangle
                    + \overline{\,   {\mathcal{P}}_{z}  \,}   {\mathcal{Q}}_{z}  \,  \langle   {\mathbf{G}}_{1},  \overline{\, {\mathbf{G}}_{2} \,} \rangle    \\
                   &=&       {\vert \,   {\mathcal{P}}_{z}  \, \vert }^{2}  \, \left(  \,  \frac{2}{\, {\vert \, g \, \vert}^{2}\,}\,  \right)
                    +   {\vert \,   {\mathcal{Q}}_{z}  \, \vert }^{2}  \,   \left(  \, 2 \, {\vert \, g \, \vert}^{2} \, \right)
                    +  {\mathcal{P}}_{z}  \overline{\,   {\mathcal{Q}}_{z}  \,}    \left(  \, -2  \, \right)
                    + \overline{\,   {\mathcal{P}}_{z}  \,}   {\mathcal{Q}}_{z}     \left(  \, -2 \, \right)   \\
                  &=& 2 \, {\left\vert \,   {\mathcal{P}}_{z}  \frac{1}{\,g\,} -   {\mathcal{Q}}_{z}       \overline{\,  g \,}             \, \right\vert}^{2} 
              =  \frac{2}{ \, {\vert \, g \, \vert}^{2} \, }  \,  {\left\vert \,   {\mathcal{P}}_{z}  -        {\vert \, g \, \vert}^{2}     {\mathcal{Q}}_{z}        \, \right\vert}^{2} >0.
   \end{eqnarray*} 
   It follows that 
  \begin{equation*} \label{Poisson 01e12} 
 {{ds}_{{}_{\Sigma}}}^{2} =   {\Lambda(z)} \, {\vert \, dz \, \vert}^{2}
 =  \frac{4}{ \, {\vert \, g \, \vert}^{2} \, }  \,  {\left\vert \,   {\mathcal{P}}_{z}  -        {\vert \, g \, \vert}^{2}     {\mathcal{Q}}_{z}        \, \right\vert}^{2}   \, {\vert \, dz \, \vert}^{2}
     \end{equation*}
We recall that ${\mathbf{X}}_{z \overline{\,z\,}} = {\mathcal{E}}_{ \overline{\,z\,}} = {\mathcal{Q}}_{z \overline{\,z\,}}  {\mathcal{N}}$.
Since the vector $\mathcal{N}$ is null, we find that the mean curvature vector $\mathbf{H}$ is also null:
    \begin{equation*} \label{Poisson 01e13} 
 \mathbf{H} :=  {\triangle}_{  {{ds}_{{}_{\Sigma}}}^{2} } \mathbf{X}  =     {\triangle}_{ {\Lambda(z)}  {\vert dz \vert}^{2} } \mathbf{X} = \frac{4}{\,{\Lambda(z)} \,}  {\mathbf{X}}_{z \overline{\,z\,}} = 
   \frac{4}{\,{\Lambda(z)} \,}  {\mathcal{Q}}_{z \overline{\,z\,}}  {\mathcal{N}}
     \end{equation*}
It remains to verify that ${\mathcal{N}}$ is the Gauss map in the sense that $\langle  {\mathbf{X}}_{z} , {\mathcal{N}} \rangle=0$. We observe that
       \begin{equation*} \label{Poisson 01e14} 
 \langle  {\mathbf{G}}_{1}, {\mathcal{N}} \rangle = \left\langle   \begin{bmatrix}   \frac{1}{\,g\,} \\[4pt]  \; \frac{i}{\,g\,}  \; \\[4pt]  1 \\[4pt]  1 \end{bmatrix}
    ,  \begin{bmatrix} 2 \, \textrm{Re} \, g  \\[4pt]   2 \, \textrm{Im} \, g \\[4pt]   -1+{\vert  \, g \, \vert}^{2}  \\[4pt]   1+{\vert  \, g \, \vert}^{2} \end{bmatrix} \right\rangle
     = 2 \frac{ \textrm{Re} \, g  + i  \textrm{Im} \, g    }{g} + \left( \,  -1+{\vert  \, g \, \vert}^{2}   \, \right) - \left( \,  1+{\vert  \, g \, \vert}^{2}   \, \right) =0,
        \end{equation*}   
and that
          \begin{equation*} \label{Poisson 01e15} 
 \langle   {\mathbf{G}}_{2}, {\mathcal{N}} \rangle =  \left\langle  \begin{bmatrix} g \\[4pt]  - i g \\[4pt]   -1  \\[4pt]   1 \end{bmatrix}
    ,  \begin{bmatrix} 2 \, \textrm{Re} \, g  \\[4pt]   2 \, \textrm{Im} \, g \\[4pt]   -1+{\vert  \, g \, \vert}^{2}  \\[4pt]   1+{\vert  \, g \, \vert}^{2} \end{bmatrix} \right\rangle
      = 2g \left( \textrm{Re} \, g  - i \, \textrm{Im} \, g  \right) - \left( \,  -1+{\vert  \, g \, \vert}^{2}   \, \right) - \left( \,  1+{\vert  \, g \, \vert}^{2}   \, \right) =0.
     \end{equation*} 
It is immediate that 
     \begin{equation*} \label{Poisson 01e16} 
 \langle  {\mathbf{X}}_{z} , {\mathcal{N}} \rangle =  \langle   {\mathcal{P}}_{z}  {\mathbf{G}}_{1} +  {\mathcal{Q}}_{z} {\mathbf{G}}_{2}, {\mathcal{N}} \rangle 
 =  {\mathcal{P}}_{z}  \langle    {\mathbf{G}}_{1}, {\mathcal{N}} \rangle  + {\mathcal{Q}}_{z}  \langle    {\mathbf{G}}_{2}, {\mathcal{N}} \rangle =0.
     \end{equation*} 
\end{proof}

   \begin{cor} [\textbf{Second Weierstrass representation for marginally trapped surfaces in ${\mathbb{L}}^{4}$}] \label{Poisson 02} 
Let $\Omega \subset {\mathbb{R}}^2 \equiv \mathbb{C}$ be a simply connected domain with the complex coordinate $z$.
Let the triple $\left(h, \, \mathcal{M}, \, \mathcal{N}\right)$ be a Weierstrass data of the second kind
The holomorphic function $h : \Omega \to \mathbb{C}-\left\{0\right\}$ and two ${\mathcal{C}}^{2}$ functions 
$\mathcal{M},  \mathcal{N}: \Omega \to \mathbb{R}$ satisfy the equation
 \begin{equation*} \label{Poisson 02a} 
  {\mathcal{M}}_{ z \overline{\,z\,} }= \left( \textrm{Re} \,  h \right)  \,  {\mathcal{N}}_{ z \overline{\,z\,} },
 \end{equation*}
 and the condition
  \begin{equation*} \label{Poisson 02b} 
{ \mathcal{M}}_{z}  \neq \left( \textrm{Re} \,  h \right)  \,  { \mathcal{N}}_{z}.
 \end{equation*}
Then, there exists a conformal parameterization $\mathbf{X} = {\mathbf{X}}_{\left(\,h, \, \mathcal{M}, \, \mathcal{N} \, \right)}:  \Omega \to {\mathbb{L}}^{4}$ of the marginally trapped surface 
$\Sigma=\mathbf{X} \left( \Omega \right)$ in Lorentz-Minkowski space ${\mathbb{L}}^{4}$ satisfying  
\begin{equation*}   \label{Poisson 02c} 
 {\mathbf{X}}_{z}=  \begin{bmatrix}  \; {\left({\mathbf{x}}_{1} \right)}_{z} \; \\[4pt]   {\left({\mathbf{x}}_{2} \right)}_{z} \\[4pt]  {\left({\mathbf{x}}_{3} \right)}_{z} \\[4pt]  {\left({\mathbf{x}}_{4} \right)}_{z} \end{bmatrix}  =  {\mathcal{M}}_{z} \begin{bmatrix} 1 \\[4pt]  -i \\[4pt]  -h \\[4pt]  h \end{bmatrix} 
  +  {\mathcal{N}}_{z} \begin{bmatrix} 0 \\[4pt]  i h \\[4pt]  \frac{1}{\,2\,} \left( 1 + h^{2} \right)  \\[4pt]  \frac{1}{\,2\,} \left( 1 - h^{2} \right) \end{bmatrix}. 
\end{equation*} 
In particular, we have ${{\mathbf{x}}_{1}}(z)=\mathcal{M}(z)$ and ${\mathbf{x}}_{3}(z)+{\mathbf{x}}_{4}(z)=\mathcal{N}(z)$, up to additive constants.
The conformal metric induced by the patch ${\mathbf{X}}(z)$ is 
     \begin{equation*} \label{Poisson 02d} 
 {{ds}_{{}_{\Sigma}}}^{2} 
   = 4 \,  {\vert  \, { \mathcal{M}}_{z}   - \left( \textrm{Re} \,  h \right)  \,  { \mathcal{N}}_{z}  \,    \vert}^{2}   \,  {\vert dz \vert}^{2}.
   \end{equation*}
\end{cor}

\begin{proof}
Using the conditions $h_{\overline{\,z\,}}=0$ and ${\mathcal{M}}_{ z \overline{\,z\,} }= \frac{\, h+ \overline{\,h\,}\,}{2} \,   {\mathcal{N}}_{ z \overline{\,z\,} }$, we deduce 
\begin{equation*}   \label{Poisson 02e01} 
 \frac{\partial}{\partial {\overline{\,z\,}} } \left(\,      h \, {\mathcal{M}}_{z} - \frac{\, h^2 \,}{2} {\mathcal{N}}_{z}  \,\right)
 =h  \left( \, \frac{\, h+ \overline{\,h\,}\,}{2} \,  {\mathcal{N}}_{ z \overline{\,z\,} } \, \right) - \frac{\, h^2 \,}{2} {\mathcal{N}}_{z \overline{\,z\,}}
 = \frac{ \, {\vert \, h \, \vert}^{2} \, }{2} {\mathcal{N}}_{ z \overline{\,z\,} } \in \mathbb{R}.
\end{equation*} 
Since the domain $\Omega$ is simply connected, by Poincar\'{e}'s Lemma, this implies the existence of an 
 ${\mathbb{R}}$-valued mapping $\mathcal{Q}$ defined on $\Omega$ satisfying the equalities
 \begin{equation*}   \label{Poisson 02e02} 
   {\mathcal{Q}}_{z}  =  h \, {\mathcal{M}}_{z} - \frac{\, h^2 \,}{2} {\mathcal{N}}_{z}  \quad \text{and} \quad 
     {\mathcal{Q}}_{ z \overline{\,z\,} } = \frac{ \, {\vert \, h \, \vert}^{2} \, }{2} {\mathcal{N}}_{ z \overline{\,z\,} }. 
\end{equation*}      
We define ${\mathcal{P}}:=\frac{1}{\,2\,} {\mathcal{N}}$ and $g:=\frac{1}{\, h \,}$. Since $h : \Omega \to \mathbb{C}-\left\{0\right\}$ is holomorphic, 
$g : \Omega \to \mathbb{C}-\left\{0\right\}$ is also holomorphic. We deduce remaining conditions in Theorem \ref{Poisson 01}:
 \begin{equation*}   \label{Poisson 02e03} 
  {\mathcal{P}}_{ z \overline{\,z\,} } = \frac{1}{\,2\,}  {\mathcal{N}}_{ z \overline{\,z\,} } = \frac{1}{\,{\vert \, h \, \vert}^{2} \,}   {\mathcal{Q}}_{ z \overline{\,z\,} }  =  {\vert \, g \, \vert}^{2}    {\mathcal{Q}}_{ z \overline{\,z\,} },  
 \end{equation*} 
 and
  \begin{equation*}  \label{Poisson 02e04} 
  {\mathcal{P}}_{z} -  {\vert \, g \, \vert}^{2} \,  {\mathcal{Q}}_{z}
  =  \frac{1}{\,2\,} {\mathcal{N}}_{z} - \frac{1}{\, h \overline{\,h\,}  \,} \left(\,  h \, {\mathcal{M}}_{z} - \frac{\, h^2 \,}{2} {\mathcal{N}}_{z}  \, \right) 
  = - \frac{1}{\,  \overline{\,h\,} \,} \left( \, { \mathcal{M}}_{z}  - \left( \textrm{Re} \,  h \right)  \,  { \mathcal{N}}_{z} \, \right) \neq 0.
\end{equation*} 
Taking the triple $\left(g, {\mathcal{P}}, {\mathcal{Q}} \right)$ in Theorem \ref{Poisson 01} yields the existence of 
$\mathbf{X} =  {\mathbf{X}}_{\left[\,g, \, \mathcal{P}, \, \mathcal{Q} \, \right]}:  \Omega \to {\mathbb{L}}^{4}$ of the marginally trapped surface 
$\Sigma=\mathbf{X} \left( \Omega \right)$ in Lorentz-Minkowski space ${\mathbb{L}}^{4}$ satisfying  
\begin{equation*}   \label{Poisson 02e05} 
  \begin{bmatrix}  \; {\left({\mathbf{x}}_{1} \right)}_{z} \; \\[4pt]   {\left({\mathbf{x}}_{2} \right)}_{z} \\[4pt]  {\left({\mathbf{x}}_{3} \right)}_{z} \\[4pt]  {\left({\mathbf{x}}_{4} \right)}_{z} \end{bmatrix}  
  =  {\mathbf{X}}_{z} =  {\mathcal{P}}_{z} \begin{bmatrix}  \frac{1}{\,g\,} \\[4pt]  \; \frac{i}{\,g\,}  \; \\[4pt]   1 \\[4pt]  1 \end{bmatrix} 
  +  {\mathcal{Q}}_{z} \begin{bmatrix} g \\[4pt]  - i g \\[4pt]   -1  \\   1 \end{bmatrix}  =  {\mathcal{M}}_{z} \begin{bmatrix} 1 \\[4pt]  -i \\[4pt]  -h \\[4pt]  h \end{bmatrix} 
  +  {\mathcal{N}}_{z} \begin{bmatrix} 0 \\[4pt]  i h \\[4pt]  \frac{1}{\,2\,} \left( 1 + h^{2} \right)  \\[4pt]  \frac{1}{\,2\,} \left( 1 - h^{2} \right) \end{bmatrix}. 
\end{equation*} 
 It follows from the equalities ${\left({\mathbf{x}}_{1} \right)}_{z} = {\mathcal{M}}_{z}$ and ${\left({\mathbf{x}}_{3} \right)}_{z} + {\left({\mathbf{x}}_{4} \right)}_{z} =  {\mathcal{N}}_{z}$ that
  \begin{equation*}   \label{Poisson 02e06} 
  {\mathbf{x}}_{1}(z)=\mathcal{M}(z) \quad \text{and} \quad {\mathbf{x}}_{3}(z) + {\mathbf{x}}_{4}(z)=\mathcal{N}(z),
  \end{equation*} 
   up to additive constants.  
We recall that $g=\frac{1}{\, h \,}$ and $  {\mathcal{P}}_{z} -  {\vert \, g \, \vert}^{2} \,  {\mathcal{Q}}_{z}
  = - \frac{1}{\,  \overline{\,h\,} \,} \left( \, { \mathcal{M}}_{z}  - \left( \textrm{Re} \,  h \right)  \,  { \mathcal{N}}_{z} \, \right)$.
We use Theorem \ref{Poisson 01} to compute the conformal metric induced by the patch ${\mathbf{X}}(z)$: 
     \begin{equation*} \label{Poisson 02e07} 
 {{ds}_{{}_{\Sigma}}}^{2} 
   = \frac{4}{ \, {\vert \, g \, \vert}^{2} \, }  \,  {\left\vert \,   {\mathcal{P}}_{z}  -        {\vert \, g \, \vert}^{2}     {\mathcal{Q}}_{z}        \, \right\vert}^{2}  \, {\vert \, dz \, \vert}^{2}
  = 4 \,  {\vert  \, { \mathcal{M}}_{z}   - \left( \textrm{Re} \,  h \right)  \,  { \mathcal{N}}_{z}  \,    \vert}^{2}   \,  {\vert dz \vert}^{2}.
   \end{equation*}
\end{proof}

\begin{rem}   \label{Poisson 02 Remarks}
  \begin{enumerate}
  \item[]
  \item[\textbf{(1)}] We assume that the two triples $\left(\,h, \, \mathcal{M}, \, \mathcal{N} \, \right)$ and $\left(g, {\mathcal{P}}, {\mathcal{Q}} \right)$ satisfy the relations
       \begin{equation*}  \label{Poisson 02r01} 
 g=\frac{1}{\, h \,},  \quad {\mathcal{P}}=\frac{1}{\,2\,} {\mathcal{N}}, \quad 
     {\mathcal{Q}}_{z}  =  h \, {\mathcal{M}}_{z} - \frac{\, h^2 \,}{2} {\mathcal{N}}_{z}.
  \end{equation*} 
The proof of Corollary \ref{Poisson 02} shows the claim that 
 ${\mathbf{X}}_{\left(\,h, \, \mathcal{M}, \, \mathcal{N} \, \right)}  = {\mathbf{X}}_{\left[\,g, \, \mathcal{P}, \, \mathcal{Q} \, \right]}$, up to an additive vector constant. In other words, the Weierstrass representation induced by the triple $\left(g, {\mathcal{P}}, {\mathcal{Q}} \right)$ in Theorem  \ref{Poisson 01}
  and  the Weierstrass representation induced by the triple $\left(h, {\mathcal{M}}, {\mathcal{N}} \right)$ in Corollary \ref{Poisson 03}
  yields the same  marginally trapped surface in ${\mathbb{L}}^{4}$, up to translations. 
\item[\textbf{(2)}] In the particular case  when ${\mathcal{M}}_{z} \equiv 0$ (or equivalently, ${\mathcal{M}}$ is constant) and ${\mathcal{N}}_{z}$ is holomorphic (or equivalently, ${\mathcal{N}}$ is harmonic), the representation in Corollary \ref{Poisson 02} for 
marginally trapped surface in ${\mathbb{L}}^{4}$ reduces to 
\begin{equation*}   \label{Poisson 02r02} 
 {\mathbf{X}}_{z} \, dz =  \begin{bmatrix}   {\left({\mathbf{x}}_{1} \right)}_{z} \\  {\left({\mathbf{x}}_{2} \right)}_{z} \\ {\left({\mathbf{x}}_{3} \right)}_{z} \\ {\left({\mathbf{x}}_{4} \right)}_{z} \end{bmatrix}  dz 
  =   {\mathcal{N}}_{z} \begin{bmatrix} 0 \\ i h \\ \frac{1}{\,2\,} \left( 1 + h^{2} \right)  \\ \frac{1}{\,2\,} \left( 1 - h^{2} \right) \end{bmatrix} dz,  
\end{equation*} 
which recovers the Weierstrass representation of the second kind (due to O. Kobayashi \cite[Corollary 1.3]{Kobayashi 1983}) for maximal surfaces in
 ${\mathbb{L}}^{3} = {\mathbb{L}}^{4} \cap \{x_{1}=0\}$ equipped with the Lorentzian metric ${dx_{2}}^{2}+{dx_{3}}^{2}-{dx_{4}}^{2}$. The induced conformal metric by the patch ${\mathbf{X}}$ is 
     \begin{equation*}  \label{Poisson 02r03} 
 {{ds}_{{}_{\Sigma}}}^{2} 
  = 4 \,  {\vert  \, { \mathcal{M}}_{z}  - \left( \textrm{Re} \,  h \right)  \,  { \mathcal{N}}_{z}  \,    \vert}^{2}   \,  {\vert \, dz \, \vert}^{2}
  = 4  \, {\left( \textrm{Re} \,  h \right)}^{2}  \, {\vert \, { \mathcal{N}}_{z} \, \vert}^{2}  \,  {\vert \, dz \, \vert}^{2}.
   \end{equation*}
  \end{enumerate}
\end{rem}

Minimal surfaces in Euclidean space ${\mathbb{R}}^{3} = {\mathbb{L}}^{4} \cap \{x_{4}=0\}$ and 
maximal surfaces in Lorentz-Minkowski space ${\mathbb{L}}^{3} = {\mathbb{L}}^{4} \cap \{x_{3}=0\}$ are 
examples of spacelike surfaces with null mean curvature vector in ${\mathbb{L}}^{4}$.  
We present a geometric representation for marginally trapped surfaces, which contains the classical Weierstrass representations for 
minimal surfaces and maximal surfaces simultaneously. The idea is to use Theorem \ref{Poisson 02} to prescribe the two height functions and the 
complexified Gauss map for the marginally trapped surface in ${\mathbb{L}}^{4}$.

  \begin{cor} [\textbf{Third Weierstrass representation for marginally trapped surfaces in ${\mathbb{L}}^{4}$}] \label{Poisson 03} 
Let $\Omega \subset {\mathbb{R}}^2  \equiv \mathbb{C}$ be a simply connected domain with the complex coordinate $z$.
Given a holomorphic function $g : \Omega \to \mathbb{C}-\left\{0\right\}$, we consider two ${\mathcal{C}}^{2}$ functions 
$\mathcal{A},  \mathcal{B}: \Omega \to \mathbb{R}$ satisfying the equation
 \begin{equation*} \label{Poisson 03a} 
 {\mathcal{A}}_{z \overline{\,z\,}} = \frac{\,  -1 +  {\vert \, g \, \vert}^{2}  \, }{\,  1 + {\vert \, g \, \vert}^{2}  \,}  \;  {\mathcal{B}}_{z \overline{\,z\,}},
 \end{equation*}
 and the condition
  \begin{equation*} \label{Poisson 03b} 
 {\mathcal{A}}_{z} \neq \frac{\,  -1 +  {\vert \, g \, \vert}^{2}  \, }{\,  1 + {\vert \, g \, \vert}^{2}  \,}  \;  {\mathcal{B}}_{z}.
 \end{equation*}
Then, there exists a conformal parameterization $\mathbf{X}={\mathbf{X}}_{\left\{\,g, \, \mathcal{A}, \, \mathcal{B} \, \right\}}:  \Omega \to {\mathbb{L}}^{4}$ of the marginally trapped surface 
$\Sigma=\mathbf{X} \left( \Omega \right)$ in Lorentz-Minkowski space ${\mathbb{L}}^{4}$ satisfying  
\begin{equation*}   \label{Poisson 03e01a} 
 {\mathbf{X}}_{z} =    \begin{bmatrix}  \; {\left({\mathbf{x}}_{1} \right)}_{z} \; \\[4pt]   {\left({\mathbf{x}}_{2} \right)}_{z} \\[4pt]  {\left({\mathbf{x}}_{3} \right)}_{z} \\[4pt]  {\left({\mathbf{x}}_{4} \right)}_{z} \end{bmatrix}  =
  {\mathcal{A}}_{z} \begin{bmatrix} \frac{1}{\,2\,}  \left( \frac{1}{\,g\,} - g \right) \\ \frac{i}{\,2\,}  \left( \frac{1}{\,g\,} + g \right)  \\ 1 \\ 0 \end{bmatrix} 
  +  {\mathcal{B}}_{z} \begin{bmatrix} \frac{1}{\,2\,}  \left( \frac{1}{\,g\,} + g \right) \\ \frac{i}{\,2\,}  \left( \frac{1}{\,g\,} - g \right)  \\ 0 \\ 1 \end{bmatrix}. 
\end{equation*} 
We have ${{\mathbf{x}}_{3}}(z)=\mathcal{A}(z)$ and ${\mathbf{x}}_{4} (z)=\mathcal{B}(z)$, up to additive constants. 
The Gauss map ${\mathcal{N}}$ of the marginally trapped surface $\Sigma$ defined by 
   \begin{equation*} \label{Poisson 03e01b} 
   {\mathcal{N}}:= \begin{bmatrix} 2 \, \textrm{Re} \, g  \\[4pt]   2 \, \textrm{Im} \, g \\[4pt]   -1+{\vert  \, g \, \vert}^{2}  \\[4pt]   1+{\vert  \, g \, \vert}^{2} \end{bmatrix}
  \end{equation*}
  satisfies the orthogonality condition $\langle  {\mathbf{X}}_{z} , {\mathcal{N}} \rangle=0$. The conformal metric induced by the patch ${\mathbf{X}}(z)$ is 
     \begin{equation*} \label{Poisson 03e01c} 
 {{ds}_{{}_{\Sigma}}}^{2} 
  =  {\left( \,   \vert \, g \, \vert + \frac{1}{\vert \, g \, \vert}   \, \right)}^{2}
   \,  { \left\vert \, {\mathcal{A}}_{z} -\frac{\,  -1 +  {\vert \, g \, \vert}^{2}  \, }{\,  1 + {\vert \, g \, \vert}^{2}  \,}  \;  {\mathcal{B}}_{z} 
  \, \right\vert}^{2}  \, {\vert \, dz \, \vert}^{2}
   \end{equation*}
\end{cor}

\begin{proof}
We take $\left(\mathcal{P}, \mathcal{Q}\right) = \left( \frac{ \, \mathcal{A} +  \mathcal{B} \, }{2},  \frac{\, - \mathcal{A} +  \mathcal{B} \, }{2} \right)$. We obtain 
      \begin{equation*} \label{Poisson 03e02} 
  {\mathcal{P}}_{z \overline{\,z\,}} -  {\vert \, g \, \vert}^{2}  \;  {\mathcal{Q}}_{z \overline{\,z\,}}
  = \frac{1}{2} 
  \left(\,  \left( 1 + {\vert \, g \, \vert}^{2}  \right)  {\mathcal{A}}_{z \overline{\,z\,}}  -  \left( -1 + {\vert \, g \, \vert}^{2}  \right)  {\mathcal{B}}_{z \overline{\,z\,}}
  \right) = 0,
    \end{equation*}
     and
       \begin{equation*} \label{Poisson 03e03} 
          {\mathcal{P}}_{z} -  {\vert \, g \, \vert}^{2} \,  {\mathcal{Q}}_{z}
          = \frac{1}{2} 
  \left(\,  \left( 1 + {\vert \, g \, \vert}^{2}  \right)  {\mathcal{A}}_{z}  -  \left( -1 + {\vert \, g \, \vert}^{2}  \right)  {\mathcal{B}}_{z}
  \right) \neq 0.
    \end{equation*}
Taking the triple $\left(g, \mathcal{P}, \mathcal{Q}\right) = \left(g, \frac{ \, \mathcal{A} +  \mathcal{B} \, }{2},  \frac{\, - \mathcal{A} +  \mathcal{B} \, }{2} \right)$ in 
Theorem \ref{Poisson 01} yields the existence of ${\mathbf{X}}_{\left[\,g, \, \mathcal{P}, \, \mathcal{Q} \, \right]}:  \Omega \to {\mathbb{L}}^{4}$ of the marginally trapped surface 
$\Sigma={\mathbf{X}}_{\left[\,g, \, \mathcal{P}, \, \mathcal{Q} \, \right]} \left( \Omega \right)$ in ${\mathbb{L}}^{4}$ satisfying 
 \begin{equation*}   \label{Poisson 03e04} 
  \frac{\partial}{\,\partial z \, } {{\mathbf{X}}_{\left[\,g, \, \mathcal{P}, \, \mathcal{Q} \, \right]} } 
 =  {\mathcal{P}}_{z} \begin{bmatrix}  \frac{1}{\,g\,} \\[4pt]  \; \frac{i}{\,g\,}  \; \\[4pt]  1 \\[4pt]  1 \end{bmatrix} 
  +  {\mathcal{Q}}_{z} \begin{bmatrix} g \\[4pt]  - i g \\[4pt]   -1  \\[4pt]   1 \end{bmatrix}.
 \end{equation*}
Setting ${\mathbf{X}}_{\left\{\,g, \, \mathcal{A}, \, \mathcal{B} \, \right\}} (z):={\mathbf{X}}_{\left[\,g, \, \mathcal{P}, \, \mathcal{Q} \, \right]}(z)$ and using $\left( {\mathcal{P}}_{z}, {\mathcal{Q}}_{z} \right) = \left( \frac{ \, {\mathcal{A}}_{z} +  {\mathcal{B}}_{z} \, }{2},  \frac{\, - {\mathcal{A}}_{z} +  {\mathcal{B}}_{z} \, }{2} \right)$, we have 
 \begin{equation*}   \label{Poisson 03e05} 
    \begin{bmatrix}  \; {\left({\mathbf{x}}_{1} \right)}_{z} \; \\[4pt]   {\left({\mathbf{x}}_{2} \right)}_{z} \\[4pt]  {\left({\mathbf{x}}_{3} \right)}_{z} \\[4pt]  {\left({\mathbf{x}}_{4} \right)}_{z} \end{bmatrix}  = 
     \frac{\partial}{\,\partial z \, } {\mathbf{X}}_{\left\{\,g, \, \mathcal{A}, \, \mathcal{B} \, \right\}}   = \frac{\partial}{\,\partial z \, } {{\mathbf{X}}_{\left[\,g, \, \mathcal{P}, \, \mathcal{Q} \, \right]} } 
 = {\mathcal{A}}_{z} \begin{bmatrix} \frac{1}{\,2\,}  \left( \frac{1}{\,g\,} - g \right) \\ \frac{i}{\,2\,}  \left( \frac{1}{\,g\,} + g \right)  \\ 1 \\ 0 \end{bmatrix} 
  +  {\mathcal{B}}_{z} \begin{bmatrix} \frac{1}{\,2\,}  \left( \frac{1}{\,g\,} + g \right) \\ \frac{i}{\,2\,}  \left( \frac{1}{\,g\,} - g \right)  \\ 0 \\ 1 \end{bmatrix}.
 \end{equation*}
 It follows from the equalities ${\left({\mathbf{x}}_{3} \right)}_{z}={\mathcal{A}}_{z}$ and  ${\left({\mathbf{x}}_{4} \right)}_{z}={\mathcal{B}}_{z}$ that
   \begin{equation*}   \label{Poisson 03e06} 
      {\mathbf{x}}_{3} (z) =  \mathcal{A} (z)  \quad \text{and} \quad   {\mathbf{x}}_{4} (z) =  \mathcal{B} (z),
   \end{equation*}
   up to additive constants. It follow from Theorem \ref{Poisson 01} that 
  \begin{equation*}   \label{Poisson 03e07} 
 {{ds}_{{}_{\Sigma}}}^{2} 
 =  \frac{4}{ \, {\vert \, g \, \vert}^{2} \, }  \,  {\left\vert \,   {\mathcal{P}}_{z}  -        {\vert \, g \, \vert}^{2}     {\mathcal{Q}}_{z}        \, \right\vert}^{2}   \, {\vert \, dz \, \vert}^{2}
 =  {\left( \,   \vert \, g \, \vert + \frac{1}{\vert \, g \, \vert}   \, \right)}^{2}
   \,  { \left\vert \, {\mathcal{A}}_{z} -\frac{\,  -1 +  {\vert \, g \, \vert}^{2}  \, }{\,  1 + {\vert \, g \, \vert}^{2}  \,}  \;  {\mathcal{B}}_{z} 
  \, \right\vert}^{2}  \, {\vert \, dz \, \vert}^{2}.
     \end{equation*}
\end{proof}

\begin{rem}     \label{Poisson 03 Remarks}
\begin{enumerate}
\item[] 
\item[\textbf{(1)}] Corollary \ref{Poisson 03} shows the method how to prescribe two height functions $ \mathcal{A}={\mathbf{x}}_{3}$, $\mathcal{B}={\mathbf{x}}_{4}$, and the 
complexified Gauss map $g$ of the marginally trapped surface in ${\mathbb{L}}^{4}$.
  \item[\textbf{(2)}] The proof of Corollary \ref{Poisson 03} indicates that ${\mathbf{X}}_{\left\{\,g, \, \mathcal{A}, \, \mathcal{B} \, \right\}}  = {\mathbf{X}}_{\left[\,g, \, \mathcal{P}, \, \mathcal{Q} \, \right]}$, up to an additive vector constant. In other words, the Weierstrass representation induced by the triple $\left(g, {\mathcal{A}}, {\mathcal{B}} \right)$ in Theorem  \ref{Poisson 03}
  and  the Weierstrass representation induced by the triple $\left(g, {\mathcal{P}}, {\mathcal{Q}} \right)$ in Corollary \ref{Poisson 01}
  yields the same  marginally trapped surface in ${\mathbb{L}}^{4}$, up to translations. 
\item[\textbf{(3)}] In the particular case when ${\mathcal{A}}_{z}$ is holomorphic and ${\mathcal{B}}_{z} \equiv 0$, the representation in Corollary \ref{Poisson 03} for 
marginally trapped surface in ${\mathbb{L}}^{4}$ reduces to 
\begin{equation*}   \label{Poisson 03r01} 
 {\mathbf{X}}_{z} \, dz =  \begin{bmatrix}   {\left({\mathbf{x}}_{1} \right)}_{z} \\  {\left({\mathbf{x}}_{2} \right)}_{z} \\ {\left({\mathbf{x}}_{3} \right)}_{z} \\ {\left({\mathbf{x}}_{4} \right)}_{z} \end{bmatrix}  
 dz    =  {\mathcal{A}}_{z} \begin{bmatrix} \frac{1}{\,2\,}  \left( \frac{1}{\,g\,} - g \right) \\ \frac{i}{\,2\,}  \left( \frac{1}{\,g\,} + g \right)  \\ 1 \\ 0 \end{bmatrix}  dz,
 \end{equation*} 
which recovers the Weierstrass representation for minimal surfaces in 
Euclidean space ${\mathbb{R}}^{3} = {\mathbb{L}}^{4} \cap \{x_{4}=0\}$. The holomorphic differential ${\mathcal{A}}_{z} dz$ is the height differential in the $x_{3}$-coordinate.
The induced conformal metric by the patch ${\mathbf{X}}$ is 
     \begin{equation*}.  \label{Poisson 03r02} 
 {{ds}_{{}_{\Sigma}}}^{2} 
  =  {\left( \,   \vert \, g \, \vert + \frac{1}{\vert \, g \, \vert}   \, \right)}^{2}
   \,  { \left\vert \, {\mathcal{A}}_{z} -\frac{\,  -1 +  {\vert \, g \, \vert}^{2}  \, }{\,  1 + {\vert \, g \, \vert}^{2}  \,}  \;  {\mathcal{B}}_{z} 
  \, \right\vert}^{2}  \, {\vert \, dz \, \vert}^{2} 
  =  {\left( \,   \vert \, g \, \vert + \frac{1}{\vert \, g \, \vert}   \, \right)}^{2} {\vert \, {\mathcal{A}}_{z} \, \vert}^{2}
  \, {\vert \, dz \, \vert}^{2}. 
   \end{equation*}
\item[\textbf{(4)}] In the particular case when ${\mathcal{A}}_{z} \equiv 0$ and ${\mathcal{B}}_{z}$ is holomorphic, the representation in Corollary \ref{Poisson 03} for 
marginally trapped surface in ${\mathbb{L}}^{4}$ reduces to 
\begin{equation*}   \label{Poisson 03e01a1} 
 {\mathbf{X}}_{z} \, dz =  \begin{bmatrix}   {\left({\mathbf{x}}_{1} \right)}_{z} \\  {\left({\mathbf{x}}_{2} \right)}_{z} \\ {\left({\mathbf{x}}_{3} \right)}_{z} \\ {\left({\mathbf{x}}_{4} \right)}_{z} \end{bmatrix}  dz 
  =   {\mathcal{B}}_{z} \begin{bmatrix} \frac{1}{\,2\,}  \left( \frac{1}{\,g\,} + g \right) \\ \frac{i}{\,2\,}  \left( \frac{1}{\,g\,} - g \right)  \\ 0 \\ 1 \end{bmatrix} dz,
\end{equation*} 
which recovers the classical Weierstrass representation for maximal surfaces in Lorentz-Minkowski space ${\mathbb{L}}^{3} = {\mathbb{L}}^{4} \cap \{x_{3}=0\}$ equipped with the
Lorentzian metric ${dx_{1}}^{2}+{dx_{2}}^{2}-{dx_{4}}^{2}$.
The holomorphic differential ${\mathcal{B}}_{z} dz$ is the height differential in the $x_{4}$-coordinate. The induced conformal metric by the patch ${\mathbf{X}}$ is 
     \begin{equation*} \label{Poisson 03e01c1} 
 {{ds}_{{}_{\Sigma}}}^{2} 
  =  {\left( \,   \frac{1 +{ \vert \, g \, \vert}^{2}  }{\vert \, g \, \vert}   \, \right)}^{2}
   \,  { \left\vert \, {\mathcal{A}}_{z} -\frac{\,  -1 +  {\vert \, g \, \vert}^{2}  \, }{\,  1 + {\vert \, g \, \vert}^{2}  \,}  \;  {\mathcal{B}}_{z} 
  \, \right\vert}^{2}  \, {\vert \, dz \, \vert}^{2} 
  =  {\left( \,   \vert \, g \, \vert - \frac{1}{\vert \, g \, \vert}   \, \right)}^{2} \, {\vert \, {\mathcal{B}}_{z} \, \vert}^{2} \, {\vert \, dz \, \vert}^{2}. 
   \end{equation*}
\end{enumerate}
\end{rem}

 \section{Geometric interpretations of  parameters} \label{parameters}
 
The purpose of this section is to reveal geometric meanings of  parameters in the elliptic, parabolic, hyperbolic deformations of our integral systems (introduced in Section \ref{integrable systems}) in terms of the geometry of marginally trapped surfaces in Lorentz-Minkowski space ${\mathbb{L}}^{4}$.

\begin{prop}[\textbf{Geometric meaning of the  parameter $\lambda$ in the parabolic deformation in Proposition \ref{FIS01}}] \label{Mparabolic}
  Let the triple $\left(g, \, \mathcal{P}, \, \mathcal{Q}\right)$ be a Weierstrass data of the first kind. Let $\lambda \in \mathbb{R}$ be a parameter constant 
 such that $g(z) \neq - \frac{1}{\,  i \lambda   \, } $ for all $z \in \Omega$. We set $g_{ {}_{\lambda} } := \frac{g}{\, 1 + i \lambda g\,}$ and ${\mathcal{P}}_{\lambda} :=\mathcal{P}$.  
 Then, the following statements hold:
 \begin{enumerate}
 \item[\textbf{(1)}] There exists a ${\mathcal{C}}^{2}$ function ${\mathcal{Q}}_{\lambda}: \Omega \to \mathbb{R}$ satisfying the integrability condition 
 \begin{equation*}  \label{Mparabolic01}
      {\left( \, {\mathcal{Q}}_{\lambda} \, \right)}_{z} = \left( \,  \frac{1}{\,g\,} + i \lambda \, \right) \left(  g    {\mathcal{Q}}_{z} - i \lambda {\mathcal{P}}_{z}      \right).  
 \end{equation*}
 Moreover, the triple $\left(g_{ {}_{\lambda} }, \,{\mathcal{P}}_{\lambda}, \, {\mathcal{Q}}_{\lambda}\right)$ becomes a Weierstrass data of the first kind. 
 \item[\textbf{(2)}] The marginally trapped surface ${\Sigma}_{\lambda}= {\mathbf{X}}_{\left[\, g_{ {}_{\lambda} }, \, {\mathcal{P}}_{\lambda}, \,{\mathcal{Q}}_{\lambda} \, \right]}  \left( \Omega \right)$ is congruent to the marginally trapped surface ${\Sigma}={\Sigma}_{0}={\mathbf{X}}_{\left[\,g, \, \mathcal{P}, \, \mathcal{Q} \, \right]} \left( \Omega \right) $ in ${\mathbb{L}}^{4}$. More concretely, 
 we have 
   \begin{equation*}  \label{Mparabolic02}
         {\mathbf{X}}_{\left[\, g_{ {}_{\lambda} }, \, {\mathcal{P}}_{\lambda}, \,{\mathcal{Q}}_{\lambda} \, \right]}
         = {\mathcal{L}}_{\lambda}^{ {}^{ \text{parabolic} }  } \left( \, {\mathbf{X}}_{\left[\,g, \, \mathcal{P}, \, \mathcal{Q} \, \right]}(z) \, \right),  
 \end{equation*}
 up to a translation in ${\mathbb{L}}^{4}$. Here, the parabolic rotation  ${\mathcal{L}}_{\lambda}^{ {}^{ \text{parabolic} }}$ in ${\mathbb{L}}^{4}$ denotes the linear transformation represented by the matrix
 \begin{equation*}  \label{Mparabolic03}
    {\mathcal{L}}_{\lambda}^{ {}^{ \text{parabolic} }} = \begin{bmatrix}
        \; 1 \; & \; 0 \; & \; 0 \; & \; 0 \; \\
        \; 0 \; & \; 1 \; & \; - \lambda \; & \;  - \lambda \; \\
        \; 0 \; & \;  \lambda \; & \; 1 - \frac{{\lambda}^{2}}{2} \; & \; - \frac{{\lambda}^{2}}{2} \; \\
        \; 0 \; & \;  - \lambda \; & \;  \frac{{\lambda}^{2}}{2} \; & \; 1+ \frac{{\lambda}^{2}}{2} \; \\       
    \end{bmatrix}.
 \end{equation*} 
 \end{enumerate}
 The item \textbf{(2)} indicates that the 
parabolic rotation  ${\mathcal{L}}_{\lambda}^{ {}^{ \text{parabolic} }}$ of the marginally trapped surface in ${\mathbb{L}}^{4}$ induces the parabolic deformation in Proposition \ref{FIS01}.  
\end{prop}
  
  \begin{proof} We need the contents of the second proof of Proposition \ref{FIS01} and the proof of Corollary \ref{Poisson 02}. 
  The deformation in the item \textbf{(1)} is proved in Proposition \ref{FIS01}. First, the second proof of Proposition \ref{FIS01} reveals that the Weierstrass triple 
  $\left(g_{ {}_{\lambda} }, \,{\mathcal{P}}_{\lambda}, \, {\mathcal{Q}}_{\lambda}\right)$ of the first kind corresponds to the Weierstrass triple of the second kind
 \begin{equation*} \label{MparabolicP01}
  \left(h_{{}_{\lambda}}, \, {\mathcal{M}}_{\lambda}, \, {\mathcal{N}}_{\lambda }\right)=\left(h + i \lambda, \, \mathcal{M}, \, \mathcal{N}\right).
 \end{equation*} 
Second, the proof of Corollary \ref{Poisson 02} (and Remark \ref{Poisson 02 Remarks}) guarantees the existence of a conformal parameterization 
${\mathbf{X}}_{\lambda} : \Omega \to {\mathbb{L}}^{4}$ of the marginally trapped surface in Lorentz-Minkowski space ${\mathbb{L}}^{4}$:
 \begin{equation*}  \label{MparabolicP02}
     {\mathbf{X}}_{\lambda} := {\mathbf{X}}_{\left[\, g_{ {}_{\lambda} }, \, {\mathcal{P}}_{\lambda}, \,{\mathcal{Q}}_{\lambda} \, \right]}
   = {\mathbf{X}}_{\left(\, h_{ {}_{\lambda} }, \, {\mathcal{M}}_{\lambda}, \,{\mathcal{N}}_{\lambda} \, \right)} 
   = {\mathbf{X}}_{\left(\,  h+ i \lambda, \, {\mathcal{M}}, \, {\mathcal{N}}  \, \right)},
 \end{equation*}
 up to an additive vector constant. Corollary \ref{Poisson 02} shows that 
   \begin{equation*}   \label{MparabolicP03}  
 {\left(   {\mathbf{X}}_{0}  \right)}_{z}  =  {\mathcal{M}}_{z} \begin{bmatrix} 1 \\[4pt]  -i \\[4pt]  -h \\[4pt]  h \end{bmatrix} 
  +  {\mathcal{N}}_{z} \begin{bmatrix} 0 \\[4pt]  i h \\[4pt]  \frac{1}{\,2\,} \left( 1 + h^{2} \right)  \\[4pt]  \frac{1}{\,2\,} \left( 1 - h^{2} \right) \end{bmatrix}. 
\end{equation*} 
and that 
  \begin{equation*}   \label{MparabolicP04} 
 {\left(   {\mathbf{X}}_{\lambda}  \right)}_{z} =  {\mathcal{M}}_{z} \begin{bmatrix} 1 \\[4pt]  -i \\[4pt]  -\left(h+ i \lambda\right) \\[4pt]  h+ i \lambda \end{bmatrix} 
  +  {\mathcal{N}}_{z} \begin{bmatrix} 0 \\[4pt]  i \left(h+ i \lambda\right)  \\[4pt]  \frac{1}{\,2\,} \left( 1 + {\left(h+ i \lambda \right)}^{2} \right)  \\[4pt]  \frac{1}{\,2\,} 
 \left(1 - {\left(h+ i \lambda\right)}^{2} \right) \end{bmatrix}.
\end{equation*} 
However, a straightforward computation shows that 
  \begin{equation*}   \label{MparabolicP05} 
   {\mathcal{M}}_{z} \begin{bmatrix} 1 \\[4pt]  -i \\[4pt]  -\left(h+ i \lambda\right) \\[4pt]  h+ i \lambda \end{bmatrix} 
  +  {\mathcal{N}}_{z} \begin{bmatrix} 0 \\[4pt]  i \left(h+ i \lambda\right)  \\[4pt]  \frac{1}{\,2\,} \left( 1 + {\left(h+ i \lambda \right)}^{2} \right)  \\[4pt]  \frac{1}{\,2\,} 
 \left(1 - {\left(h+ i \lambda\right)}^{2} \right) \end{bmatrix}
 =  {\mathcal{L}}_{\lambda}^{ {}^{ \text{parabolic} }} \left( \,   {\mathcal{M}}_{z} \begin{bmatrix} 1 \\[4pt]  -i \\[4pt]  -h \\[4pt]  h \end{bmatrix} 
  +  {\mathcal{N}}_{z} \begin{bmatrix} 0 \\[4pt]  i h \\[4pt]  \frac{1}{\,2\,} \left( 1 + h^{2} \right)  \\[4pt]  \frac{1}{\,2\,} \left( 1 - h^{2} \right) \end{bmatrix}  \, \right).
\end{equation*} 
Integrating the differential equation
  \begin{equation*}   \label{MparabolicP06} 
 {\left(   {\mathbf{X}}_{\lambda}  \right)}_{z}  =  {\mathcal{L}}_{\lambda}^{ {}^{ \text{parabolic} }} \left( \,  {\left(   {\mathbf{X}}_{0}  \right)}_{z} \right)
 =    {\left(  {\mathcal{L}}_{\lambda}^{ {}^{ \text{parabolic} }}  {\mathbf{X}}_{0}  \right)}_{z},
 \end{equation*} 
  yields that,  up to an additive vector constant,   
  \begin{equation*}   \label{MparabolicP07} 
         {\mathbf{X}}_{\lambda}(z) = {\mathcal{L}}_{\lambda}^{ {}^{ \text{parabolic} }  } \left( \,    {\mathbf{X}}_{0}(z)  \, \right).  
 \end{equation*}
  \end{proof}

  \begin{prop}[\textbf{Geometric meaning of the  parameter $\tau$ in the elliptic deformation in Proposition \ref{FIS02}}] \label{Melliptic}
     Let the triple $\left(h, \, \mathcal{M}, \, \mathcal{N}\right)$ be a Weierstrass data of the second kind. Given a parameter $\tau \in \mathbb{R}$,
 we set $h_{ {}_{\tau} } := e^{-i \tau} h$ and ${\mathcal{N}}_{\tau} :=\mathcal{N}$. Then, the following statements hold:
 \begin{enumerate}
 \item[\textbf{(1)}] 
  Then, there exists a ${\mathcal{C}}^{2}$ function ${\mathcal{M}}_{\tau}: \Omega \to \mathbb{R}$ satisfying the integrability condition 
 \begin{equation*}  \label{Melliptic01}
      {\left( \, {\mathcal{M}}_{\tau} \, \right)}_{z} = e^{i \tau}  \, {\mathcal{M}}_{ z } - \frac{\,  e^{i \tau} -  e^{- i \tau}   \,}{2}  \, h \,  {\mathcal{N}}_{ z }. 
 \end{equation*}
 Moreover, the triple $\left(h_{ {}_{\tau} }, \,{\mathcal{M}}_{\tau}, \, {\mathcal{N}}_{\tau}\right)$ becomes a Weierstrass data of the second kind. 
 \item[\textbf{(2)}] The marginally trapped surface ${\Sigma}_{\tau}= {\mathbf{X}}_{\left(h_{ {}_{\tau} }, \,{\mathcal{M}}_{\tau}, \, {\mathcal{N}}_{\tau}\right)}  \left( \Omega \right)$ is congruent to the marginally trapped surface ${\Sigma}={\Sigma}_{0}={\mathbf{X}}_{\left(\,h, \, \mathcal{M}, \, \mathcal{N} \, \right)} \left( \Omega \right) $ in ${\mathbb{L}}^{4}$. More concretely, 
 we have 
   \begin{equation*}   \label{Melliptic02}
         {\mathbf{X}}_{\left(h_{ {}_{\tau} }, \,{\mathcal{M}}_{\tau}, \, {\mathcal{N}}_{\tau}\right)}
         = {\mathcal{L}}_{\tau}^{ {}^{ \text{elliptic} }  } \left( \, {\mathbf{X}}_{\left(\,h, \, \mathcal{M}, \, \mathcal{N} \, \right)}(z) \, \right),  
 \end{equation*}
 up to a translation in ${\mathbb{L}}^{4}$. Here, the elliptic rotation  ${\mathcal{L}}_{\tau}^{ {}^{ \text{elliptic} }}$ in ${\mathbb{L}}^{4}$ denotes the linear transformation represented by the matrix
 \begin{equation*}   \label{Melliptic03}
    {\mathcal{L}}_{\tau}^{ {}^{ \text{elliptic} }} = \begin{bmatrix}
        \; \cos \tau \; & \; -  \sin \tau \; & \; 0 \; & \; 0 \; \\
          \; \sin \tau \; & \;  \cos \tau \; & \; 0 \; & \; 0 \; \\
        \; 0 \; & \; 0 \; & \; 1 \; & \; 0 \; \\
        \; 0 \; & \; 0 \; & \; 0 \; & \; 1 \; \\     
    \end{bmatrix}.
 \end{equation*} 
 \end{enumerate}
 The item \textbf{(2)} indicates that the 
elliptic rotation  ${\mathcal{L}}_{\lambda}^{ {}^{ \text{elliptic} }}$ of the marginally trapped surface in ${\mathbb{L}}^{4}$ induces the elliptic deformation in Proposition \ref{FIS02}.  
  \end{prop}
  
    \begin{proof}  
    First, the second proof of Proposition \ref{FIS02} reveals that the Weierstrass triple  $\left(h_{ {}_{\tau} }, \,{\mathcal{M}}_{\tau}, \, {\mathcal{N}}_{\tau}\right)$ of the first kind   
corresponds to the Weierstrass triple of the second kind
 \begin{equation*} \label{MellipticP01}
  \left({g}_{{}_{\tau}}, \, {\mathcal{P}}_{\tau}, \, {\mathcal{Q}}_{\tau}\right)=\left(e^{i \tau} g, \, \mathcal{P}, \, \mathcal{Q}\right).
 \end{equation*} 
Second, the proof of Corollary \ref{Poisson 02} (and Remark \ref{Poisson 02 Remarks}) guarantees the existence of a conformal parameterization 
${\mathbf{X}}_{\tau} : \Omega \to {\mathbb{L}}^{4}$ of the marginally trapped surface in Lorentz-Minkowski space ${\mathbb{L}}^{4}$:
 \begin{equation*}  \label{MellipticP02}
     {\mathbf{X}}_{\tau} := {\mathbf{X}}_{\left(\,h_{ {}_{\tau} }, \,{\mathcal{M}}_{\tau}, \, {\mathcal{N}}_{\tau}\,\right)} =  {\mathbf{X}}_{ 
       \left[\,{g}_{{}_{\tau}}, \, {\mathcal{P}}_{\tau}, \, {\mathcal{Q}}_{\tau}\,\right]}  = {\mathbf{X}}_{ \left[\,e^{i \tau} g, \, \mathcal{P}, \, \mathcal{Q}\,\right] }
 \end{equation*}
 up to an additive vector constant. Corollary \ref{Poisson 02} shows that 
   \begin{equation*}   \label{MellipticP03}  
 {\left(   {\mathbf{X}}_{0}  \right)}_{z}   =  {\mathcal{P}}_{z} \begin{bmatrix}  \frac{1}{\,g\,} \\[4pt]  \; \frac{i}{\,g\,}  \; \\[4pt]  1 \\[4pt]  1 \end{bmatrix} 
  +  {\mathcal{Q}}_{z} \begin{bmatrix} g \\[4pt]  - i g \\[4pt]   -1  \\[4pt]   1 \end{bmatrix}   
\end{equation*} 
and that 
  \begin{equation*}   \label{MellipticP04} 
 {\left(   {\mathbf{X}}_{\tau}  \right)}_{z} = {\mathcal{P}}_{z} \begin{bmatrix}  \frac{1}{\,e^{i \tau} g\,} \\[4pt]  \; \frac{i}{\,e^{i \tau} g\,}  \; \\[4pt]  1 \\[4pt]  1 \end{bmatrix} 
  +  {\mathcal{Q}}_{z} \begin{bmatrix} e^{i \tau} g \\[4pt]  - i e^{i \tau} g \\[4pt]   -1  \\[4pt]   1 \end{bmatrix}  
   =  {\mathcal{L}}_{\tau}^{ {}^{ \text{elliptic} }} \left( \, {\left( {\mathbf{X}}_{0}  \right)}_{z} \, \right)  =  {\left(  {\mathcal{L}}_{\tau}^{ {}^{ \text{elliptic} }}  {\mathbf{X}}_{0}  \right)}_{z}.
\end{equation*} 
Integrating this differential equation yields that,  up to an additive vector constant,   
  \begin{equation*}  \label{MellipticP05} 
       \left(   {\mathbf{X}}_{\tau}  \right)(z) = {\mathcal{L}}_{\tau}^{ {}^{ \text{elliptic} }}  {\mathbf{X}}_{0} (z).  
 \end{equation*}
    \end{proof}

   \begin{prop}[\textbf{Geometric meaning of the  parameter $\eta$ in the hyperbolic deformation in Proposition \ref{FIS03A}}] \label{Mhyperbolic}
  Let the triple $\left(g, \, \mathcal{P}, \, \mathcal{Q}\right)$ be a Weierstrass data of the first kind. Let $\eta \in \mathbb{R}$ be a parameter constant. 
 Then, the following statements hold:
 \begin{enumerate}
 \item[\textbf{(1)}] The triple $\left({g}_{ {}_{\eta} }, \, {\mathcal{P}}_{\eta}, \, {\mathcal{Q}}_{\eta}\right)
 :=\left( e^{\eta} g, \, e^{\eta} {\mathcal{P}}, \, e^{-\eta} \mathcal{Q}  \right)$ becomes a Weierstrass data of the first kind. 
 \item[\textbf{(2)}] The marginally trapped surface ${\Sigma}_{\eta}= {\mathbf{X}}_{\left[\, g_{ {}_{\eta} }, \, {\mathcal{P}}_{\eta }, \,{\mathcal{Q}}_{\eta} \, \right]}  \left( \Omega \right)$ is congruent to the marginally trapped surface ${\Sigma}={\Sigma}_{0}={\mathbf{X}}_{\left[\,g, \, \mathcal{P}, \, \mathcal{Q} \, \right]} \left( \Omega \right) $ in ${\mathbb{L}}^{4}$. More concretely, 
 we have 
   \begin{equation*}  \label{Mhyperbolic02}
         {\mathbf{X}}_{\left[\, g_{ {}_{\eta} }, \, {\mathcal{P}}_{\eta}, \,{\mathcal{Q}}_{\eta} \, \right]}
         = {\mathcal{L}}_{\eta}^{ {}^{ \text{hyperbolic} }  } \left( \, {\mathbf{X}}_{\left[\,g, \, \mathcal{P}, \, \mathcal{Q} \, \right]}(z) \, \right),  
 \end{equation*}
 up to a translation in ${\mathbb{L}}^{4}$. Here, the hyperbolic rotation  ${\mathcal{L}}_{\eta}^{ {}^{ \text{hyperbolic} }}$ in ${\mathbb{L}}^{4}$ denotes the linear transformation represented by the matrix
 \begin{equation*}  \label{Mhyperbolic03}
    {\mathcal{L}}_{\lambda}^{ {}^{ \text{hyperbolic} }} = \begin{bmatrix}
        \; 1 \; & \; 0 \; & \; 0 \; & \; 0 \; \\
    \; 0 \; & \; 1 \; & \; 0 \; & \; 0 \; \\
          \; 0 \; & \; 0 \; & \; \cosh \eta \; & \; \sinh \eta \; \\
      \; 0 \; & \; 0 \; & \; \sinh \eta \; & \; \cosh \eta \; \\   
    \end{bmatrix}.
 \end{equation*} 
 \end{enumerate}
 The item \textbf{(2)} indicates that the 
hyperbolic rotation  ${\mathcal{L}}_{\eta}^{ {}^{ \text{hyperbolic} }}$ of the marginally trapped surface in ${\mathbb{L}}^{4}$ induces the hyperbolic deformation in 
Proposition \ref{FIS03A}.  
  \end{prop}
  
    \begin{proof}  The deformation in the item \textbf{(1)} is proved in Proposition \ref{FIS03A}. 
 Corollary \ref{Poisson 02} guarantees the existence of a conformal parameterization 
${\mathbf{X}}_{\eta}: \Omega \to {\mathbb{L}}^{4}$  of the marginally trapped surface in Lorentz-Minkowski space ${\mathbb{L}}^{4}$:
\begin{equation*}   \label{MhyperbolicP03}  
    {\mathbf{X}}_{\eta} := {\mathbf{X}}_{\left[\, g_{ {}_{\eta} }, \, {\mathcal{P}}_{\eta}, \,{\mathcal{Q}}_{\eta} \, \right]} = 
    {\mathbf{X}}_{\left[\, e^{\eta} g, \, e^{\eta} {\mathcal{P}}, \, e^{-\eta} \mathcal{Q} \, \right]}.
\end{equation*} 
 Corollary \ref{Poisson 02} implies that 
\begin{equation*}   \label{MhyperbolicP04}  
  {\left( {\mathbf{X}}_{0} \right)}_{z}  
 =  {\mathcal{P}}_{z} \begin{bmatrix}  \frac{1}{\,g\,} \\[4pt]  \; \frac{i}{\,g\,}  \; \\[4pt]  1 \\[4pt]  1 \end{bmatrix} 
  +  {\mathcal{Q}}_{z} \begin{bmatrix} g \\[4pt]  - i g \\[4pt]   -1  \\[4pt]   1 \end{bmatrix}
\end{equation*} 
and that 
  \begin{equation*}   \label{MhyperbolicP05} 
  {\left( {\mathbf{X}}_{\eta} \right)}_{z}  
 =  e^{\eta} {\mathcal{P}}_{z} \begin{bmatrix}  \frac{1}{\,e^{\eta} g\,} \\[4pt]  \; \frac{i}{\,e^{\eta} g\,}  \; \\[4pt]  1 \\[4pt]  1 \end{bmatrix} 
  + e^{-\eta} {\mathcal{Q}}_{z} \begin{bmatrix} e^{\eta} g \\[4pt]  - i e^{\eta} g \\[4pt]   -1  \\[4pt]   1 \end{bmatrix} 
     =  {\mathcal{L}}_{\eta}^{ {}^{ \text{hyperbolic} }} \left( \, {\left( {\mathbf{X}}_{0}  \right)}_{z} \, \right)  =  {\left(  {\mathcal{L}}_{\eta}^{ {}^{ \text{hyperbolic} }}  {\mathbf{X}}_{0}  \right)}_{z}.
\end{equation*} 
Integrating this differential equation yields that,  up to an additive vector constant,   
  \begin{equation*}   \label{MhyperbolicP06} 
         {\mathbf{X}}_{\eta}(z) = {\mathcal{L}}_{\eta}^{ {}^{ \text{hyperbolic} }  } \left( \,    {\mathbf{X}}_{0}(z)  \, \right).  
 \end{equation*}
  \end{proof}

 \section{Examples of marginally trapped surfaces} \label{deformed examples}

  We use our integrable systems and Weierstrass representations to construct explicit examples of the marginally trapped surface in ${\mathbb{L}}^{4}$ with nowhere vanishing mean curvature vector.  We adopt the complex coordinate $z=u+iv$ with $u, v \in \mathbb{R}$. To use Corollary \ref{Poisson 02}, we need to prepare the Weierstrass triple $\left(h, \mathcal{M}, \mathcal{N} \right)$ of the second kind, which solves the system 
    \begin{equation*}   \label{DE00p01} 
 \begin{cases}
    h_{\overline{\,z\,} } =0, \\
   {\mathcal{M}}_{ z \overline{\,z\,} }= \left( \textrm{Re} \,  h \right)  \,  {\mathcal{N}}_{ z \overline{\,z\,} }, \\ 
  {\mathcal{M}}_{z} - \left( \textrm{Re} \,  h \right)  \,  { \mathcal{N}}_{z} \neq 0.
\end{cases}   
   \end{equation*}  
  We take the normalization with the nowhere vanishing holomorphic function 
   \begin{equation*}   \label{DE00p02} 
   h(z) =  e^{i z} = e^{-v} \left(\cos u + i \sin u \right) \neq 0.
   \end{equation*}
Then, we need to find a pair $\left( \mathcal{M}(u, v), \,  \mathcal{N}(u, v) \right)$ of $\mathbb{R}$-valued functions solving the Poisson equation of the form
   \begin{equation*}   \label{DE00p03} 
     {\mathcal{M}}_{uu} +    {\mathcal{M}}_{vv} = e^{-v} \cos u \left(   {\mathcal{N}}_{uu} +    {\mathcal{N}}_{vv} \right).  
   \end{equation*}
We observe that this \emph{linear} equation admits the following two independent solutions:
    \begin{eqnarray}   \label{DE00p04} 
        \left( {\mathcal{M}}_{1} (u, v), \,  {\mathcal{N}}_{1}(u, v) \right) &=& \left( \sinh u \sin u, \, e^v \cosh u \right),  \\ 
    \left( {\mathcal{M}}_{2} (u, v), \,  {\mathcal{N}}_{2}(u, v) \right) &=&  \left(  - \cos u \cos v, \, e^v \sin v \right).
   \end{eqnarray}
   
Using these two solutions and the linear structure of the above Poisson equation, we construct families of marginally trapped surfaces with nowhere vanishing mean curvature vector.
 
 \begin{exmp}[\textbf{A family of marginally trapped surfaces in ${\mathbb{L}}^{4}$ connecting from a CMC-$1$ surface in ${\mathbb{H}}^{3}$ to a CMC-$1$ surface in  ${\mathbb{S}}^{3}_{1}$}] \label{Bryant01} 
   We construct a one parameter family ${ \left\{ {\Sigma}_{\theta} \right\}} _{\theta  \in \left[0, \frac{\, \pi \,}{2} \right]}$ of marginally trapped surfaces with nowhere vanishing mean curvature vector field. Moreover, the surface ${\Sigma}_{\theta}$ lies in the hypersurface in ${\mathbb{L}}^{4}$:
    \begin{equation*}   \label{DE01p01} 
       {{x}_{1}}^{2} +    {{x}_{2}}^{2} +    {{x}_{3}}^{2} -    {{x}_{4}}^{2} = - \cos \left( 2 \theta \right).
   \end{equation*} 
 Given $\theta \in \left[0, \frac{\, \pi \,}{2} \right]$, we choose an open domain ${\Omega}_{\theta} \subset {\mathbb{R}}^{2}$ 
 satisfying the condition
   \begin{equation*}   \label{DE01p02} 
       \cos \theta \cosh u + \sin \theta \cos v \neq 0,
   \end{equation*}
 for all $\left(u, v\right) \in {\Omega}_{\theta}$. For instance, if $\theta \in \left[0, \frac{\, \pi \,}{4}\right)$ we take   
    \begin{equation*}   \label{DE01p03} 
 {\Omega}_{\theta} = {\mathbb{R}}^{2}.
   \end{equation*}
 For $\theta = \frac{\, \pi \,}{4}$, we take 
     \begin{equation*}   \label{DE01p04} 
 {\Omega}_{\frac{\, \pi \,}{4}} = \left\{ \,  \left( u, v  \right)  \in {\mathbb{R}}^{2}  \; \vert \; u>0, \; v \in \left( -\pi, \pi \right) \, \right\}.
   \end{equation*}
 For  $\theta = \frac{\, \pi \,}{2}$, we take
     \begin{equation*}   \label{DE01p05} 
 {\Omega}_{\frac{\, \pi \,}{2}} = \left\{ \,  \left( u, v  \right)  \in {\mathbb{R}}^{2}  \; \vert \; u \in \mathbb{R}, \; v \in \left( - \frac{\, \pi \,}{2},  \frac{\, \pi \,}{2} \right) \, \right\}.
   \end{equation*}
  Given $\theta \in \left[0, \frac{\, \pi \,}{2} \right]$, we define the triple $\left( {h}_{\theta} , \, {\mathcal{M}}_{\theta}, \, {\mathcal{N}}_{\theta} \right)$ by 
  \begin{eqnarray*}   \label{DE01p06} 
             {h}_{\theta} (z) &=& e^{i z} = e^{-v} \left(\cos u + i \sin u \right) \neq 0, \\
      {\mathcal{M}}_{\theta} \left( z \right) &=& \cos \theta \sinh u \sin u - \sin \theta \cos u \cos v, \\   
     {\mathcal{N}}_{\theta} \left( z \right) &=&  {e}^{v} \left( \cos \theta \cosh u + \beta \sin v  \right).        
  \end{eqnarray*} 
 The triple $\left(h, \mathcal{M}, \mathcal{N} \right) = \left( {h}_{\theta} , \, {\mathcal{M}}_{\theta}, \, {\mathcal{N}}_{\theta} \right)$ is a Weierstrass data of the second kind: 
   \begin{eqnarray*}    \label{DE01p07} 
             h_{  \overline{\,z\,} } &=& 0, \\
        4 {\mathcal{N}}_{ z \overline{\,z\,} } &=& e^{v}  \left( 2 \cos \theta \cosh u + \sin \theta \cos v \right), \\    
          4 {\mathcal{M}}_{ z \overline{\,z\,} } &=&  \cos u \left( 2 \cos \theta \cosh u + \sin \theta \cos v \right)
       =     \left( \textrm{Re} \,  h \right)  \,\left( 4 {\mathcal{N}}_{ z \overline{\,z\,} } \right), \\   
  {\mathcal{M}}_{z} - \left( \textrm{Re} \,  h \right)  \,  { \mathcal{N}}_{z} 
  &=& \frac{i}{\,2\,} e^{-iu} \left(  \cos \theta \cosh u + \sin \theta \cos v \right) \neq 0.        
  \end{eqnarray*} 
Applying Corollary \ref{Poisson 02} to the triple $\left( {h}_{\theta} , \, {\mathcal{M}}_{\theta}, \, {\mathcal{N}}_{\theta} \right)$, we can find 
 a conformal patch
${\mathbf{X}}_{\theta} := {\mathbf{X}}_{\left( {h}_{\theta} , \, {\mathcal{M}}_{\theta}, \, {\mathcal{N}}_{\theta} \right)} : {\Omega}_{\theta} \to {\mathbb{L}}^{4}$ of the marginally trapped surface 
${\Sigma}_{\theta}$ in ${\mathbb{L}}^{4}$:
   \begin{equation*}     \label{DE01p08} 
{\mathbf{X}}_{\theta} (u, v)  =   \cos \theta \, {\mathbf{X}}_{0}  (u, v)  +  \sin \theta \,   {\mathbf{X}}_{\frac{\, \pi \,}{2}}  (u, v), 
\end{equation*} 
where we introduce
   \begin{equation*}     \label{DE01p09} 
 {\mathbf{X}}_{0}  (u, v)   =  \begin{bmatrix}  \sinh u \sin u  \\  \sinh u \cos u    \\   \cosh u \sinh v  \\   \cosh u \cosh v    \end{bmatrix} \quad  \text{and} \quad
  {\mathbf{X}}_{\frac{\, \pi \,}{2}}  (u, v)   =  \begin{bmatrix}   - \cos u \cos v  \\  \sin u \cos v    \\  \cosh v   \sin v \\   \sinh v   \sin v    \end{bmatrix}.
 \end{equation*} 
The metric of the spacelike surface ${\Sigma}_{\theta}$ induced by the patch ${\mathbf{X}}_{\theta}$ is
    \begin{equation*}   \label{DE01p10} 
  {\Lambda}_{\theta}(z)  {\vert dz \vert}^{2}  
   =  {\left( \cos \theta \cosh u + \sin \theta \cos v  \right)}^{2} \left( \, du^2 + dv^2 \, \right).
   \end{equation*}
We claim that the mean curvature vector $\mathbf{H}$ vanishes nowhere on the surface ${\Sigma}_{\theta}$: 
    \begin{equation*}   \label{DE01p11} 
   \mathbf{H} =  
      {\triangle}_{    {\Lambda}_{\theta}(z)    {\vert dz \vert}^{2}    }  {\mathbf{X}}_{\theta}  
   = \frac{1}{\,   {\Lambda}_{\theta}(z)  \,} \left(\,   {\left( \,  {\mathbf{X}}_{\theta}   \, \right)}_{uu} +   {\left( \, {\mathbf{X}}_{\theta}    \, \right)}_{vv}    \, \right). 
          \end{equation*}
Recall that $ \cos \theta \cosh u + \sin \theta \cos v \neq 0$ for all $\left(u, v \right) \in {\Omega}_{\theta}$. The null vector
   \begin{equation*}   \label{DE01p12} 
    {\left( \,  {\mathbf{X}}_{\theta}   \, \right)}_{uu} +   {\left( \, {\mathbf{X}}_{\theta}    \, \right)}_{vv}
      =  2 \left( \cos \theta \cosh u + \sin \theta \cos v  \right)  \begin{bmatrix}   \cos u  \\  \; - \sin u \;    \\   \sinh v  \\  \cosh v   \end{bmatrix}
   \end{equation*}
    vanishes nowhere because it is not possible that $\cos u$ and $-\sin u$ can be zero simultaneously.
  It now remains to show that the the spacelike surface ${\Sigma}_{\theta}$, defined by the patch
     \begin{equation*}   \label{DE01p13} 
{\mathbf{X}}_{\theta}  =   \cos \theta \, {\mathbf{X}}_{0}   +  \sin \theta \,   {\mathbf{X}}_{\frac{\, \pi \,}{2}}
   \end{equation*} 
   lies in the hypersurface $   {{x}_{1}}^{2} +    {{x}_{2}}^{2} +    {{x}_{3}}^{2} -    {{x}_{4}}^{2} = - \cos \left( 2 \theta \right)$ in ${\mathbb{L}}^{4}$. We have 
  \begin{eqnarray*}   \label{DE01p14} 
       \langle {\mathbf{X}}_{0}, {\mathbf{X}}_{0} \rangle &=& 
        \left\langle   \begin{bmatrix}  \sinh u \sin u  \\  \sinh u \cos u    \\   \cosh u \sinh v  \\   \cosh u \cosh v    \end{bmatrix}, 
       \begin{bmatrix}  \sinh u \sin u  \\  \sinh u \cos u    \\   \cosh u \sinh v  \\   \cosh u \cosh v    \end{bmatrix}  \right\rangle =-1, \\
             \langle {\mathbf{X}}_{0}, {\mathbf{X}}_{\frac{\, \pi \,}{2}} \rangle &=&  \left\langle   \begin{bmatrix}  \sinh u \sin u  \\  \sinh u \cos u    \\   \cosh u \sinh v  \\   \cosh u \cosh v    \end{bmatrix}, 
     \begin{bmatrix}   - \cos u \cos v  \\  \sin u \cos v    \\  \cosh v   \sin v \\   \sinh v   \sin v    \end{bmatrix}  \right\rangle =0,  \\
            \langle {\mathbf{X}}_{\frac{\, \pi \,}{2}}, {\mathbf{X}}_{\frac{\, \pi \,}{2}} \rangle &=&  \left\langle   \begin{bmatrix}   - \cos u \cos v  \\  \sin u \cos v    \\  \cosh v   \sin v \\   \sinh v   \sin v    \end{bmatrix}, 
     \begin{bmatrix}   - \cos u \cos v  \\  \sin u \cos v    \\  \cosh v   \sin v \\   \sinh v   \sin v    \end{bmatrix}  \right\rangle=1.
   \end{eqnarray*} 
We conclude that 
    \begin{equation*}   \label{DE01p15} 
        \langle {\mathbf{X}}_{\theta}, {\mathbf{X}}_{\theta}  \rangle
        =   \langle \cos \theta \, {\mathbf{X}}_{0}   +  \sin \theta \,   {\mathbf{X}}_{\frac{\, \pi \,}{2}}, \cos \theta \, {\mathbf{X}}_{0}   +  \sin \theta \,   {\mathbf{X}}_{\frac{\, \pi \,}{2}} \rangle
        =   - {\cos}^{2} \theta +  {\sin}^{2} \theta = - \cos \left( 2 \theta \right).
   \end{equation*} 
 \end{exmp} 

\begin{rem}  \label{DE01 Remarks} 
\item[\textbf{(1)}] Taking $\theta=0$ in Example \ref{Bryant01} yields a marginally trapped surface ${\Sigma}_{0} \subset {\mathbb{L}}^{4}$ parameterized by 
    \begin{equation*}    \label{DE01RA01} 
 {\mathbf{X}}_{0}  (u, v)   =  \begin{bmatrix}  \sinh u \sin u  \\  \sinh u \cos u    \\   \cosh u \sinh v  \\   \cosh u \cosh v    \end{bmatrix}, \quad (u, v) \in {\mathbb{R}}^{2}.
    \end{equation*} 
The surface ${\Sigma}_{0}$ is a Bryant surface \cite{Bryant 1987}, which lies in the hyperboloid model of the three dimensional hyperbolic space ${\mathbb{H}}^{3} \subset {\mathbb{L}}^{4}$ 
given by the hypersurface
     \begin{equation*}    \label{DE01RA02} 
  {{x}_{1}}^{2} +    {{x}_{2}}^{2} +    {{x}_{3}}^{2} -    {{x}_{4}}^{2} = -1,
     \end{equation*}  
 and has mean curvature one in ${\mathbb{H}}^{3}$.  The surface ${\Sigma}_{0}$ is called 
\emph{catenoid cousin} (\cite[Example 2]{Bryant 1987} and \cite[Example 5.6]{DFS2021}) in ${\mathbb{H}}^{3}$ in the sense that it is locally isometric to a catenoid, which is a rotational surface with zero mean curvature in  ${\mathbb{R}}^{3}$. The conformal 
 metric on the surface ${\Sigma}_{0}$ induced by the patch ${\mathbf{X}}_{0}  (u, v)$ is 
     \begin{equation*}    \label{DE01RA03} 
 {\cosh}^{2} u \left( du^2 + dv^2 \right), 
  \end{equation*} 
 which is the conformal metric induced by the 
 patch  $ {\mathbf{X}}_{\text{cat}}  (u, v)$  of the catenoid ${\Sigma}_{\text{cat}}$ in  the three dimensional Euclidean space ${\mathbb{R}}^{3} =  {\mathbb{L}}^{4} \cap \{ x_{4}=0 \}$:
    \begin{equation*}    \label{DE01RA04} 
 {\mathbf{X}}_{\text{cat}}   (u, v)   =  \begin{bmatrix}  \cosh u \sin v  \\  \cosh u \cos v    \\  u  \\   0   \end{bmatrix}, \quad (u, v) \in {\mathbb{R}}^{2}.
    \end{equation*} 
\item[\textbf{(2)}] Taking $\theta=\frac{\pi}{2}$ in Example \ref{Bryant01} yields a marginally trapped surface ${\Sigma}_{\frac{\, \pi \,}{2}} \subset {\mathbb{L}}^{4}$ parameterized by 
    \begin{equation*}    \label{DE01RB01} 
 {\mathbf{X}}_{\frac{\, \pi \,}{2}}  (u, v)   =  \begin{bmatrix}   - \cos u \cos v  \\  \sin u \cos v    \\  \cosh v   \sin v \\   \sinh v   \sin v    \end{bmatrix}, \quad (u, v) 
 \in  \mathbb{R} \times  \left( - \frac{\, \pi \,}{2},  \frac{\, \pi \,}{2} \right).
    \end{equation*} 
The surface ${\Sigma}_{\frac{\, \pi \,}{2}}$ is lies in the model of the three dimensional de-Sitter space ${\mathbb{S}}^{3}_{1} \subset {\mathbb{L}}^{4}$ 
given by the hypersurface
     \begin{equation*}    \label{DE01RB02} 
  {{x}_{1}}^{2} +    {{x}_{2}}^{2} +    {{x}_{3}}^{2} -    {{x}_{4}}^{2} = 1,
     \end{equation*}  
 and has mean curvature one in ${\mathbb{S}}^{3}_{1}$. The surface ${\Sigma}_{\frac{\, \pi \,}{2}}$ can be viewed as a \emph{hyperbolic catenoid cousin} in ${\mathbb{S}}^{3}_{1}$.
 The conformal metric on the surface ${\Sigma}_{\frac{\, \pi \,}{2}}$ induced by the patch $ {\mathbf{X}}_{\frac{\, \pi \,}{2}}  (u, v)$ is 
     \begin{equation*}    \label{DE01RB03} 
 {\cos}^{2} v \left( du^2 + dv^2 \right), 
  \end{equation*} 
 which is the conformal metric induced by the 
 patch  $ {\mathbf{X}}_{\text{hcat}}  (u, v)$  of the hyperbolic catenoid ${\Sigma}_{\text{hcat}}$ in 
 the three dimensional Lorentz-Minkowski space ${\mathbb{L}}^{3} = {\mathbb{L}}^{4} \cap \{ x_{1}=0 \}$:
    \begin{equation*}    \label{DE01RB04} 
 {\mathbf{X}}_{\text{hcat}}   (u, v)   =  \begin{bmatrix}  0  \\  v   \\  \cos v \sinh u  \\    \cos v \cosh u   \end{bmatrix}, \quad (u, v) \in  \mathbb{R} \times  \left( - \frac{\, \pi \,}{2},  \frac{\, \pi \,}{2} \right).
    \end{equation*} 
   \item[\textbf{(3)}] Taking $\theta=\frac{\pi}{4}$ in Example \ref{Bryant01} yields a marginally trapped surface ${\Sigma}_{\frac{\, \pi \,}{4}} \subset {\mathbb{L}}^{4}$ parameterized by 
    \begin{equation*}    \label{DE01RC01} 
      {\mathbf{X}}_{\frac{\, \pi \,}{4}}  (u, v)   
 =  \frac{1}{\,\sqrt{\,2\,}\,}  \begin{bmatrix}  \sinh u \sin u  \\  \sinh u \cos u    \\   \cosh u \sinh v  \\   \cosh u \cosh v    \end{bmatrix} 
  + \frac{1}{\,\sqrt{\,2\,}\,}   \begin{bmatrix}   - \cos u \cos v  \\  \sin u \cos v    \\  \cosh v   \sin v \\   \sinh v   \sin v    \end{bmatrix}, 
   \quad  (u, v) \in  \mathbb{R} \times  \left( - \frac{\, \pi \,}{2},  \frac{\, \pi \,}{2} \right).
    \end{equation*}  
The surface ${\Sigma}_{\frac{\, \pi \,}{4}}$ is null in the sense that it lies in the light cone $  {{x}_{1}}^{2} +    {{x}_{2}}^{2} +    {{x}_{3}}^{2} -    {{x}_{4}}^{2} = 0$. 
\end{rem}
 
 \begin{exmp}[\textbf{Two parameter deformations of a catenoid cousin in ${\mathbb{H}}^{3}$ to marginally trapped surfaces 
 in ${\mathbb{L}}^{4}$}]  \label{Bryant02}
 We take $\Omega={\mathbb{R}}^{2} \equiv \mathbb{C}$ and the complex coordinate. 
We produce a two parameter family of marginally trapped surfaces with nowhere vanishing mean curvature vector field.
 Given a pair $\left(\alpha, \beta\right)$ of real constants such that 
  $ { \alpha}^2 + {\beta}^2 <1$, we define the triple $\left({h}_{\left(\alpha, \beta\right)} , \, {\mathcal{M}}_{\left(\alpha, \beta\right)} , \, {\mathcal{N}}_{\left(\alpha, \beta\right)} \right)$ by 
  \begin{eqnarray*}   \label{DE02p03} 
             {h}_{\left(\alpha, \beta\right)} (z) &=& e^{i z} = e^{-v} \left(\cos u + i \sin u \right) \neq 0, \\
             {\mathcal{M}}_{\left(\alpha, \beta\right)}  \left( z \right) &=&  \alpha u + \beta v + \sinh u \sin u, \\   
             {\mathcal{N}}_{\left(\alpha, \beta\right)}  \left( z \right) &=&  {e}^{v} \cosh u.        
  \end{eqnarray*} 
 The triple $\left(h, \mathcal{M}, \mathcal{N} \right) = \left({h}_{\left(\alpha, \beta\right)} , \, {\mathcal{M}}_{\left(\alpha, \beta\right)} , \, {\mathcal{N}}_{\left(\alpha, \beta\right)} \right)$ is a Weierstrass data of the second kind: 
   \begin{eqnarray*}   \label{DE02p04} 
             h_{  \overline{\,z\,} } &=& 0, \\
          4 {\mathcal{M}}_{ z \overline{\,z\,} } &=& 2 \cos u \cosh u =  \left(  e^{-v} \cos u \right)   \left( 2 e^{v} \cosh u \right) = 
           \left( \textrm{Re} \,  h \right)  \,\left( 4 {\mathcal{N}}_{ z \overline{\,z\,} } \right), \\   
  {\mathcal{M}}_{z} - \left( \textrm{Re} \,  h \right)  \,  { \mathcal{N}}_{z} 
  &=& \left(  \alpha + \cosh u \sin u \right) + i \left( - \beta + \cosh u \cos u \right) \neq 0.        
  \end{eqnarray*} 
 The proof of the last statement is given as follows. Assume to the contrary that there exists a pair $\left(u, v\right)$ of real numbers such that
 $  {\mathcal{M}}_{z} - \left( \textrm{Re} \,  h \right)  \,  { \mathcal{N}}_{z}=0$, or equivalently, that 
   \begin{equation*}   \label{DE02p05} 
      \cosh u \sin u  = -  \alpha \quad \text{and} \quad  \cosh u \cos u =  \beta.
  \end{equation*} 
 Combining the assumption $a^2 + b^2 <1$ and these two conditions yields 
    \begin{equation*}   \label{DE02p06} 
     1 >  { \alpha}^2 + {\beta}^2 = {\left( \cosh u \sin u  \right)}^{2} + {\left( \cosh u \cos u \right)}^{2} = {\cosh}^{2} u \geq 1,
  \end{equation*} 
  which is a contradiction. Applying Corollary \ref{Poisson 02} to the triple $\left(h, \, \mathcal{M}, \, \mathcal{N}\right)$, we can find 
 a conformal patch
${\mathbf{X}}_{\left(\alpha, \beta\right)} := {\mathbf{X}}_{\left(\,  h, \, {\mathcal{M}}, \, {\mathcal{N}}  \, \right)} : {\mathbb{R}}^{2} \to {\mathbb{L}}^{4}$ of the marginally trapped surface 
${\Sigma}_{\left(\alpha, \beta\right)}$ in ${\mathbb{L}}^{4}$:
   \begin{equation*}   \label{DE02p07} 
{\mathbf{X}}_{\left(\alpha, \beta\right)} (u, v) 
 =   \alpha \begin{bmatrix}   u  \\  v    \\   - e^{-v} \sin u  \\  e^{-v} \sin u   \end{bmatrix} 
  + \beta  \begin{bmatrix}   v  \\  -u   \\   e^{-v} \cos u  \\   - e^{-v} \cos u   \end{bmatrix} 
 + \begin{bmatrix}  \sinh u \sin u  \\  \sinh u \cos u    \\   \cosh u \sinh v  \\   \cosh u \cosh v    \end{bmatrix}.
\end{equation*} 
The metric of the spacelike surface ${\Sigma}_{\left(\alpha, \beta\right)}$ induced by the patch ${\mathbf{X}}_{\left(\alpha, \beta\right)}$ is
    \begin{equation*}   \label{DE02p08} 
  {\Lambda}_{\left(\alpha, \beta\right)}(z)  {\vert dz \vert}^{2}  
   = \left( \,  {\left(  \alpha + \cosh u \sin u \right)}^{2}  + {\left( - \beta + \cosh u \cos u \right)}^{2}   \, \right) \left( \, du^2 + dv^2 \, \right).
   \end{equation*}
   Finally, we claim that the mean curvature vector $\mathbf{H}$ vanishes nowhere on the surface ${\Sigma}_{\left(\alpha, \beta\right)}$: 
     \begin{equation*}   \label{DE02p09} 
   \mathbf{H} =  
      {\triangle}_{    {\Lambda}_{\left(\alpha, \beta\right)}(z)    {\vert dz \vert}^{2}    }  {\mathbf{X}}_{\left(\alpha, \beta\right)}  
   = \frac{1}{\,   {\Lambda}_{\left(\alpha, \beta\right)}(z)  \,} \left(\,   {\left( \,  {\mathbf{X}}_{\left(\alpha, \beta\right)}   \, \right)}_{uu} +   {\left( \, {\mathbf{X}}_{\left(\alpha, \beta\right)}    \, \right)}_{vv}    \, \right). 
          \end{equation*}
 However, the null vector
   \begin{equation*}   \label{DE02p10} 
    {\left( \,  {\mathbf{X}}_{\left(\alpha, \beta\right)}   \, \right)}_{uu} +   {\left( \, {\mathbf{X}}_{\left(\alpha, \beta\right)}    \, \right)}_{vv}
      =  2 \cosh u \begin{bmatrix}   \cos u  \\  \; - \sin u \;    \\   \sinh v  \\  \cosh v   \end{bmatrix}
   \end{equation*}
    vanishes nowhere because it is not possible that $\cos u$ and $-\sin u$ can be zero simultaneously.
  \end{exmp}

\bigskip

\end{document}